\documentclass[12pt]{amsart}

\usepackage[english]{babel}

\usepackage{mathrsfs}
\usepackage{amsmath,amssymb}
\usepackage{epigraph}
\usepackage[all]{xy}
\usepackage{bm}
\usepackage{mathptmx}

\usepackage{amsthm}
\newcommand{\Ph}{\varphi}

\newcommand{\Dc}{\mathbf{D}_{\mathrm{cris}}}
\newcommand{\La}{\varLambda}
\newcommand{\Dd}{\mathbf{D}_{\mathrm {dR}}}

\newcommand{\Dst}{\mathbf{D}_{\mathrm{st}}}

\newcommand{\RG}{\mathbf{R}\Gamma}

\newcommand{\res}{\mathrm{res}}
\newcommand{\R}{\mathbf{R}}

\newcommand{\A}{\mathbf{A}}

\newcommand{\bD}{\mathbf{D}}

\newcommand{\cl}{\mathrm{cl}}

\newcommand{\Tot}{\mathrm{Tot}}
\newcommand{\Iw}{\mathrm{Iw}}
\newcommand{\Bd}{\mathbf{B}_{\mathrm{dR}}}
\newcommand{\Bc}{\mathbf{B}_{\mathrm{cris}}}

\newcommand{\Hom}{\mathrm{Hom}}
\newcommand{\Ext}{\mathrm{Ext}}

\newcommand{\F}{\mathrm{Fil}}
\newcommand{\g}{\gamma}
\newcommand{\Ind}{\mathrm{Ind}}

\newcommand{\Ddagrig}{\mathbf{D}^{\dagger}_{\mathrm{rig}}}

\newcommand{\iso}{\overset{\sim}{\rightarrow}}

\newcommand{\Rep}{\mathbf{Rep}}
\newcommand{\Gal}{\mathrm{Gal}}

\newcommand{\CR}{\mathcal{R}}
\newcommand{\CDcris}{\mathscr{D}_{\mathrm{cris}}}

\newcommand{\ep}{\varepsilon}

\newcommand{\Qp}{\mathbf Q_p}
\newcommand{\Zp}{\mathbf Z_p}
\newcommand{\Q}{\mathbf Q}
\newcommand{\CH}{\mathscr H}
\newcommand{\BrdA}{\widetilde{\mathbf B}_{\textrm{rig},A}^{\dagger}}
\newcommand{\BrdAr}{\widetilde{\mathbf B}_{\textrm{rig},A}^{\dagger, r}}
\newcommand{\BrdArp}{{\widetilde{\mathbf B}}_{\textrm{rig},A}^{\dagger, r^{1/p}}}
\newcommand{\CDst}{\mathcal D_{\textrm{st}}}

\newcommand{\gr}{\text{\rm gr}}

\newcommand{\CO}{\mathcal O}
\newcommand{\DdagrigA}{\bD^{\dagger}_{\textrm{rig},A}}

\newcommand\Fr{\textrm{Fr}}

\newcommand\dR{\textrm{dR}}
\newcommand\st{\textrm{st}}
\newcommand\cris{\textrm{cris}}

\newcommand\BQ{\mathbf Q}
\newcommand{\W}{\mathbf W}
\newcommand{\bM}{\mathbf M}
\newcommand\im{\mathrm{Im}}

\newcommand\cyc{\mathrm{cyc}}

\newcommand\SC{\mathbf R\Gamma}
\newtheorem*{mainconj}{Main Conjecture} 
\newtheorem*{definition}{Definition}
\newtheorem{mytheorem}{Theorem}
\newtheorem{myproposition}{Proposition}[]
\newtheorem*{leopoldt}{Weak Leopoldt Conjecture}
\newtheorem{mylemma}{Lemma}
\newtheorem*{conjecture}{Conjecture}

\makeindex             

\begin{document}
\title{Selmer Complexes and  the $P$-adic Hodge theory}
\author{Denis Benois}
\address{Institut de Math\'ematiques,
Universit\'e de  Bordeaux, 351, cours de la Lib\'eration  33405
Talence, France} 
\email{denis.benois@math.u-bordeaux1.fr}

%
%
\maketitle
\tableofcontents

\section*{Introduction}

\subsection{The Selmer group} In this paper we discuss some applications
of   $p$-adic Hodge theory to
the algebraic formalism of Iwasawa theory.
Fix an odd prime number $p.$
For every finite set $S$ of primes containing $S$ we denote by ${\BQ}^{(S)}/\BQ$
the maximal algebraic extension of $\Q$ unramified outside $S\cup\{\infty\}$ and
set $G_{\BQ,S}=\Gal ( {\BQ}^{(S)}/\BQ).$  Let $G_{\BQ_v}=\Gal (\overline{\BQ}_v/\BQ_v)$ 
denote the  decomposition group
at $v$ and $I_v$ its inertia subgroup. 

Let $M$ be a pure motive over $\BQ.$
The complex $L$-function $L(M,s)$ is a Dirichlet series which converges for $s\gg 0$
and is expected to admit a meromorphic continuation on the whole $\mathbf C$
with a functional equation of the form 
\[
\Gamma (M,s)L(M,s)=\varepsilon (M,s)\Gamma (M^*(1),-s)L(M^*(1),-s).
\]
where $\Gamma (M,s)$ is a product of  $\Gamma$-factors explicitly defined in terms of the Hodge structure
of $M$ and $\varepsilon (V,s)$ is a factor of the form $\ep (M,s)=ab^s$ 
defined in terms of local constants (see for example \cite{FP}).

Let $V$ denote the $p$-adic realisation of $M.$
We consider $V$ as a $p$-adic representation of the Galois group $G_{\BQ,S}$
where $S$ contains $p,$ $\infty$ and the  places where $M$ has  bad reduction.
We will write $H^*_S(\BQ,V)$ and $H^*(\BQ_v,V)$ for the continuous cohomology 
of $G_{\BQ,S}$  and $G_{\Q_v}$ respectively.  
 The Bloch--Kato Selmer group $H^1_f(\Q,V)$ is defined as
\begin{equation}
\label{definition of Bloch-Kato Selmer}
H^1_f(\BQ,V)=\ker \left (H^1_S(\BQ,V)\rightarrow \underset{v\in S}{\bigoplus}
\frac{H^1(\BQ_v,V)}{H^1_f(\BQ_v,V)}\right )
\end{equation}
where the "local conditions" $H^1_f(\BQ_v,V)$ are given by 
\begin{equation}
\label{Bloch-Kato local conditions}
H^1_f(\BQ_v,V)=\begin{cases} 
\ker (H^1(\BQ_v,V)\rightarrow H^1(I_v,V))  &\text{\rm if $v\neq  p$},
\\
\ker (H^1(\BQ_p,V) \rightarrow H^1(\BQ_p,V\otimes\Bc))  &\text{\rm if $v= p$} 
\end{cases}
\end{equation}
(see \cite{BK}).
Here  $\Bc$ is  the ring of crystalline periods of Fontaine \cite{Fo94a}.
Beilinson's conjecture (in the formulation of Bloch and Kato \cite{BK}) relates 
the rank of
$H^1_f(\BQ,V^*(1))$  to the order of vanishing of $L(M,s)$, namely
one expects that
\[
\text{\rm ord}_{s=0}L(M,s)=\dim_{\Qp}H^1_f(\Q,V^*(1))-
\dim_{\Qp}H^0(\Q,V^*(1)).
\]

\subsection{$p$-adic $L$-functions}
 Assume that $V$ satisfies the following conditions (see Section 2.1 below).

{\bf C1)} $H^0_S(\Q,V)=H^0_S(\Q,V^*(1))=0$.

{\bf C2)} The restriction of $V$ on the decomposition group at $p$ is semistable in the sense
of Fontaine \cite{Fo94b}. 
We denote by  $\Dst (V)$ the semistable module associated to $V.$ Recall that $\Dst (V)$ is a filtered
$\Qp$-vector space equipped with a linear Frobenius 
$\Ph \,:\,\Dst (V)\rightarrow \Dst (V)$ and a monodromy operator
$N\,:\,\Dst (V)\rightarrow \Dst (V)$ such that $N\Ph=p\Ph N.$ Let $\Dc (V)=\Dst (V)^{N=0}.$

{\bf C3)} $\Dc (V)^{\Ph=1}=0.$
\newline
\,

We define the notion of a regular submodule of $\Dst (V)$ which plays the key role
in this paper. It was first introduced by Perrin-Riou  \cite{PR95} in the crystalline
case. See also \cite{Ben11} and \cite{Ben14}. 
Eventually replacing $M$ by $M^*(1)$ we will assume that 
$M$ is of weight $wt (M)\leqslant -1.$
We consider two cases:
\newline
\,

{\textbf{The weight $\leqslant -2$ case.}} Let $wt (M)\leqslant -2.$ We expect that in
this case the localization map
\[
H^1_f(\Q,V) \rightarrow H^1_f(\Qp,V)
\]
is injective. The condition {\bf C3)} implies that the Bloch--Kato exponential map
\[
\exp_V\,:\,\Dst (V)/\F^0\Dst (V)\rightarrow H^1_f(\Qp,V)
\]
 is an isomorphism and we denote
by $\log_{V}\,:\,H^1_f(\Qp,V)\rightarrow  \Dst (V)/\F^0\Dst (V)$ its inverse. 
A $(\Ph,N)$-submodule $D$  of 
$\Dst (V)$ is regular if $D\cap \F^0\Dst (V)=\{0\}$ and the composition of the localisation map with $\log_V$ induces an isomorphism
\[
r_{V,D}\,:\,H^1_f(\Q,V) \rightarrow \Dst (V)/(\F^0\Dst (V)+D).
\]
Note that the map  $r_{V,D}$ should be closely related to the syntomic regulator.
\newline
\,

{\bf The weight $-1$ case.} Let $wt (M)=-1.$ In this case we say that 
 a $(\Ph,N)$-submodule
$D$ of $\Dst (V)$ is regular if 
\[
\Dst (V)=D\oplus \F^0\Dst (V) 
\]
as vector spaces. Each regular $D$ gives  a splittings of the Hodge filtration
on $\Dst (V)$ and therefore  defines a $p$-adic height pairing  
\[
\left <\,,\,\right >_{V,D}\,:\,H^1_f(\Q,V)\times H^1_f(\Q,V^*(1)) \rightarrow \Qp
\]
which is expected to be non-degenerate \cite{Ne92}, \cite{NeSC}, \cite{PR92}.

The theory of Perrin-Riou \cite{PR95} suggests that to each regular $D$ 
one can associate a $p$-adic $L$-function $L_p(M,D,s)$  interpolating rational parts of 
special values $L(M,s).$  The interpolation formula relating special values 
of the complex and the $p$-adic $L$-functions at $s=0$ should have the form
\[
\frac{L^*_p(M,D,0)}{\Omega_p(M,D)}=\mathcal E(V,D)\frac{L^*(M,0)}{\Omega_{\infty}(M)}.
\]
Here $\Omega_{\infty}(M)$ is the irrational part of $L^*(M,0)$ predicted by Beilinson--
\linebreak
Deligne's
conjectures and $\Omega_p(M,D)$ is essentially the  determinant of $r_{V,D}$ 
(in the weight $\leqslant -2$ case) or the determinant of $\left <\,,\,\right >_{V,D}$
(in the weight $-1$ case). Finally $\mathcal E(V,D)$ is an Euler-like factor. 
Its explicit form in the crystalline case was conjectured by Perrin-Riou \cite{PR95}, Chapitre 4
(see also \cite{Cz00} Conjecture 2.7).
Namely,  let $E_p(M,X)=\det (1-\Ph X\vert \Dc (V)).$
Note that  $E_p(M,p^{-s})$ is  the Euler factor of $L(M,s)$ at $p.$
One expects that 
\[
\mathcal E(V,D)=E_p(M,1)\,\det \left (\frac{1-p^{-1}\Ph^{-1}}{1-\Ph} \left \vert \right. D\right )
\qquad \text{\rm if $V$ is crystalline at $p$.}
\] 

\subsection{Selmer complexes}  We introduce basic notation of the Iwasawa theory. Let $\mu_{p^n}$ denote the group
of $p^n$-th roots of unity and $\Q^{\cyc}=\underset{n=1}{\overset\infty
{\cup}} \Q(\mu_{p^n}).$ We fix a system $(\zeta_{p^n})_{n\geqslant 0}$ of primitive $p^n$th
roots of unity such that $\zeta_{p^n+1}^p=\zeta_{p^n}$ for all $n.$
The Galois group $\Gamma =\Gal (\Q^{\cyc}/\Q)$ is canonically isomorphic 
to $\Zp^*$ {\it via} the cyclotomic character
\[
\chi\,:\,\Gamma \rightarrow \Zp^* ,\qquad g (\zeta_{p^n})=\zeta_{p^n}^{\chi (g)}\qquad
\text{\rm for all $g\in \Gamma$}.
\]
Let $\Delta =\Gal (\BQ (\mu_p)/\Q)$ and 
 $\Gamma_0=\Gal (\Q^{\cyc} /\BQ (\mu_p)).$ 
We denote by $F_{\infty}=\Q (\mu_{p^\infty})^{\Delta}$ the cyclotomic $\Zp$-extension
of $\Q$ and set  $F_n=\Q(\mu_{p^{n+1}})^{\Delta}, $ $n\in \mathbf N.$
Fix a generator $\g\in \Gamma$ and
set $\g_0=\g^{p-1}\in \Gamma_0.$ Let $\Lambda$ denote  the Iwasawa algebra $\Zp[[\Gamma_0]].$ 
The choice of $\g$ fixes an isomorphism $\Lambda\simeq \Zp[[X]]$ such that
$\g_0\mapsto 1+X.$  
Let $\CH$ denote the ring  of power series that converge on the open unit disc.
We consider $\Lambda$ as a subring of $\CH$ {\it via} the isomorphism $\Lambda \simeq \Zp[[X]].$
Note that $\CH$ is the large 
Iwasawa algebra  introduced in \cite{PR94}.

Let $T$ be $\Zp$-lattice of $V$ stable under the action of $G_{\Q,S}.$ 
Let  $T\otimes \Lambda^{\iota}$ denote the tensor product  $T\otimes \Lambda$ equipped
with the diagonal action of $G_{\Q,S}$ and the action of $\Lambda$ given by
the canonical involution $\iota \,:\,\Lambda \rightarrow \Lambda$
\[
g (a\otimes \lambda)=a\otimes \iota (g)(\lambda) =a\otimes g^{-1}\lambda,\qquad g\in \Gamma.
\]
Let $C^{\bullet}(G_{\Q,S}, T\otimes \Lambda^{\iota})$ (respectively  $C^{\bullet}(G_{\BQ_v}, T\otimes \Lambda^{\iota})$)
be the complex of continuous cochains
of $G_{\Q,S}$ (respectively $G_{\BQ_v}$) with coefficients in $T\otimes \Lambda^{\iota}.$
We denote by $\RG_{\Iw,S}(\Q,T)$ and $\RG_{\Iw,S}(\Q_v,T)$  these
complexes viewed as objects of  the bounded derived category
$D^{\mathrm b}(\La)$ of $\Lambda$-modules. Set 
\linebreak
$H^i_{\Iw,S}(\Q,T)=\R^i\Gamma_{\Iw,S}(\Q,T)$ and $H^i_{\Iw}(\Q_v,T)=\R^i\Gamma_{\Iw}(\Q_v,T).$
From 
\linebreak
Shapiro's lemma it follows that
$$
H^i_{\Iw,S}(\Q,T)=\underset{n}{\varprojlim} H^i_S(F_n,T), \qquad
H^i_{\Iw}(\Q_v,T)=\underset{n}{\varprojlim} H^i_S(F_{n,v},T)
$$
as $\Lambda$-modules. In \cite{NeSC} Nekov\'a\v r considers  diagrams of the form

\begin{equation}
\xymatrix{
\RG_{\Iw,S}(\Q,T) \ar[r] &\underset{v\in S}\bigoplus
\RG_{\Iw}(\Q_v,T)\\
&\underset{v\in S}\bigoplus U^{\bullet}_v(T). \ar[u]} 
\end{equation}
The cone of this diagram can be viewed as the derived version of Selmer groups 
where the complexes  $U^{\bullet}_v(T)$ play the role of local conditions. 
For $v\neq p$ the natural choice of $U^{\bullet}_v(T)$ is to put 
$
U^{\bullet}_v(T)=\RG_{\Iw,f}(\Q_v,T)
$ 
where 
\[
\RG_{\Iw,f}(\Q_v,T)=\left [T^{I_v}\otimes \Lambda^{\iota}\xrightarrow{\mathrm{Fr}_v-1}
T^{I_v}\otimes \Lambda^{\iota}\right ]
\] 
where $\text{\rm Fr}_v$ denotes the geometric Frobenius and the terms are placed in degrees $0$ and $1.$ Note that the cohomology groups of 
\[
\RG_{\Iw,f}(\Q_v,T)\otimes^{\mathbf L}_{\Lambda}\Qp=
\left [V^{I_v} \xrightarrow{\mathrm{Fr}_v-1}
V^{I_v}\right ]
\]
are $H^0(\Q_v,V)$ and $H^1_f(\Q_v,V).$  The local conditions
at $p$ are more delicate to define.  
First assume that $V$ satisfies the Panchishkin condition i.e. the restriction 
of $V$ on the decomposition group at $p$  has a  subrepresentation $F^+V\subset V$ such that
\[
\Dd (V)=\Dd (F^+V)\oplus   
\F^0\Dd (V)   
\]
as vector spaces.  Set $U^{\bullet}_p (T)=\RG_{\Iw}(\Qp,F^+T).$
Then $U^{\bullet}_p(T)$ is the derived version of Greenberg's local conditions
\cite{Gr89} and the cohomology of Nekov\'a\v r's Selmer complex is closely
related to the Pontriagin dual of  Greenberg's Selmer group \cite{NeSC}. If we assume in addition that 
\linebreak
$\Dc (V^*(1))^{\Ph=1}=0$ then 
$H^1_f(\Qp,V)=\ker (H^1(\Qp,V)\rightarrow H^1(I_p,V/F^+V))$
and  one can compare Greenberg's and Bloch--Kato's Selmer groups.
\linebreak
Roughly speaking, in this case different definitions lead to 
pseudo-
\linebreak
isomorphic $\Lambda$-modules which have therefore the same
characteristic series in the case they are $\Lambda$-torsion (see \cite{Fl90}, \cite{NeSC}, Chapter 9 and \cite{Och}).
If $\Dc (V^*(1))^{\Ph=1}\neq 0$ the situation is more complicated. 
The analytic counterpart of this problem is the phenomenon
of extra zeros of $p$-adic $L$-
\linebreak
functions studied in \cite{Gr94}, \cite{Ben11},  \cite{Ben14}.  

We no longer assume that $V$ satisfies the Panchishkin condition.
The theory of $(\Ph,\Gamma)$-modules 
(see \cite{Fo90}, \cite{CC2}, \cite{Ber02}) associates to $V$ a 
finitely generated free module $\Ddagrig (V)$ over the Robba ring $\CR$
equipped with
a semilinear actions of the group $\Gamma$ and a Frobenius $\Ph$ which commute with each other.  
The category of $(\Ph,\Gamma)$-modules has a nice cohomology theory whose formal properties  are very close to 
properties of  local Galois cohomology \cite{H1}, \cite{H2},
\cite{Li}. In particular, $H^*(\Ddagrig (V))$ is canonically isomorphic
\footnote{Up to the choice of a generator of $\Gamma$.}
to the continuous Galois cohomology $H^*(\Qp,V).$ Moreover, to each $(\Ph,\Gamma)$-module
$\bD$ one can associate the complex
\[
\RG_{\Iw}(\bD)=\left [\bD \xrightarrow{\psi-1}\bD \right ]^{\Delta=0}
\]
where $\psi$ is the left inverse to $\Ph$ (see \cite{CC1})  and the first term is placed
in degree $1$.  We will write  $H_{\Iw}^*(\bD)$ for the cohomology of $\RG_{\Iw}(\bD).$
 The action of $\Gamma_1$ induces 
a natural structure of $\CH$-module  on $\bD^{\Delta=0}.$  From a general result of Pottharst 
( \cite{PoAF}, Theorem 2.8)
it follows that 
\[
H^i_{\Iw}(\Qp,T)\otimes_{\Lambda}\CH  \iso H^i_{\Iw}(\Ddagrig (V))
\]
as $\CH$-modules. 

The approach to Iwasawa theory discussed in this paper is based on the observation
that the $(\Ph,\Gamma)$-module of a semistable representation looks like
an ordinary Galois representation, in particular it is a 
successive extension of $(\Ph,\Gamma)$-modules of rank $1.$
This was first pointed out by Colmez  in \cite{Cz08} where  
the structure of trianguline  $(\Ph,\Gamma)$-modules of rank $2$ over $\Qp$
was studied in detail. Therefore in the non ordinary setting
we can  adopt Greenberg's approach  working with $(\Ph,\Gamma)$-modules instead
$p$-adic representations. This idea was used in \cite{BC} and \cite{Bel}  
to study Selmer groups in families and in \cite{Ben11}, \cite{Ben13},
\cite{Ben14} to study
extra-zeros of $p$-adic $L$-functions. Pottharst \cite{Po13} started the
general theory of Selmer complexes in this setting and  related it
to Perrin-Riou's theory \cite{PoCIT}.

\subsection{The Main Conjecture} Fix  a regular submodule  $D$ of $\Dst (V).$ By \cite{Ber08} one can associate 
to $D$  a canonical  $(\Ph,\Gamma)$-submodule $\bD$  of $\Ddagrig (V).$ 
Consider the   diagram 
\[
\xymatrix{
\RG_{\Iw,S}(\Q,T)\otimes^{\mathbf L}_{\Lambda}\CH \ar[r] &\underset{v\in S}\bigoplus
\RG_{\Iw}(\Q_v,T)\otimes^{\mathbf L}_{\Lambda}\CH\\
&\underset{v\in S}\bigoplus U^{\bullet}_v(V,D) \ar[u]}
\]
in the derived category
of $\CH$-modules,
where the local conditions $U^{\bullet}_v(V,D)$ are
\[
U^{\bullet}_v(V,D)=\begin{cases} \RG_{\Iw,f}(\Q_v,T)\otimes^{\mathbf L}_{\Lambda}\CH
&\text{\rm if $v\neq \,p$}\\
\RG_{\Iw}(\bD) &\text{\rm if $v=p.$}
\end{cases}
\]

Consider the Selmer complex associated to this data
\begin{multline}
\SC_{\Iw}(V,D)=\mathrm{cone}\,\left (\left (
\RG_{\Iw,f}(\Q_v,T)\otimes^{\mathbf  L}_{\Lambda}\CH\right )\bigoplus 
\left (\underset{v\in S}{\bigoplus} U^{\bullet}_v(V,D)\right ) \rightarrow  \right.\\
\left. \underset{v\in S}\bigoplus
\RG_{\Iw}(\Q_v,T)\otimes^{\mathbf  L}_{\Lambda}\CH
\right )[-1].
\end{multline}

To simplify notation define
\[
H_{\Iw}^i (V,D)=\mathbf{R}^i\Gamma_{\Iw} (V,D).
\]

We propose the following conjecture.

\begin{mainconj} Let $M/\Q$ be a pure motive of weight $\leqslant -1$
which does not contain submotives of the form $\Q (m).$
Assume that the $p$-adic realisation $V$ of $M$ satisfies the conditions {\bf C1-3)} above.  Let $D$ be a regular submodule of $\Dst (V).$ Then 

i) $H_{\Iw}^i (V,D)=0$ for $i\neq 2.$

ii) $H_{\Iw}^2 (V,D)$ is a coadmissible
\footnote{See Section 1.6 for the definition of a coadmissible module}
 torsion $\CH$-module and 
\[
\mathrm{char}_{\CH} \left(H_{\Iw}^2 (V,D)\right )= (f_{D})
\]
where $L_p(M,D,s)=f_{D}(\chi (\g_0)^s-1).$
\end{mainconj}

\noindent 
{\bf Remarks} 1) Since  $\CH^*=\La [1/p]^*,$ 
our Main Conjecture determines 
$f_D$ up to multiplication by a unit in $\La [1/p].$

2) Assume that $M$ is critical and that $V$
is ordinary at $p$. Then the ordinarity filtration $(F^iV)_{i\in \mathbf Z}$ provides the canonical 
regular module $D=\Dst (F^1V).$ 
It is expected that in this situation $f_D\in \Lambda.$
In \cite{Gr89}, Greenberg   defined 
a cofinitely generated $\Lambda$-module (Greenberg's  
Selmer group)
\linebreak  
$S (F_\infty,V^*(1)/T^*(1))$ and conjectured that
its Pontriagin dual 
\linebreak
$S (F_\infty,V^*(1)/T^*(1))^{\wedge}$ is 
related to the $p$-adic function $L_p(V,D,s)$ as follows
\[
\mathrm{char}_{\La}S (F_\infty,V^*(1)/T^*(1))^{\wedge}=f_{D}\La
\]
Using the results of \cite{NeSC} one can check that 
this agrees with our Main conjecture, but  Greenberg's conjecture
is more precise because it  determines
$f_{D}$ up to multiplication by a unit in $\La.$ See \cite{PoCIT} and
Section 2.5 for more detail. 

3) Assume that $V$ is crystalline at $p$. In \cite{PR95}, 
for any regular $D$ Perrin-Riou defined a free $\Lambda$-module
$\mathbf L_{\Iw}(D,V)$ together with a canonical trivialisation
\[
i_{V,D}\,\,:\,\,\mathbf L_{\Iw}(D,V) \overset{\sim}{\rightarrow} 
\CH
\]
See also \cite{Ben14} for the interpretation of Perrin-Riou's theory in terms  of Selmer complexes. The Main Conjecture in the formulation of
Perrin--Riou says that
\[
i_{V,D} (\mathbf L_{\Iw}(D,V))=\Lambda f_D. 
\]
In \cite{PoCIT}, Pottharst proved that
\[i_{V,D} (\mathbf L_{\Iw}(D,V))\otimes_{\Zp}\Qp=
\mathrm{char}_{\CH}H_{\Iw}^2 (V,D) 
\]
i.e. for crystalline representations  our conjecture is compatible
with Perrin-Riou's theory.

4) Let $f=\underset{n=1}{\overset{\infty}\sum}a_nq^n$ be a newform 
of weight $k$ 
on $\Gamma_0(N)$  
and let $W_f$ be the $p$-adic representation associated to $f.$
Then $W_f$ can be seen as the $p$-adic realisation of a pure motive 
$M_f$ of weight $k-1.$ The representation $W_f$ is crystalline 
if $(p,N)=1$ and semistable non crystalline if $p\parallel N.$
Fix  $1\leqslant m\leqslant k-1.$ Then the motive $M_f(m)$ is 
critical  and we denote by
 $V=W_f(m)$  its $p$-adic realisation. 
The subspace $\F^0\Dst (V)$ is one dimensional.
Eventually extending scalars to some $E/\Qp$
 we can write
$\Dst (V)=Ee_{\alpha}+Ee_{\beta}$ where $\Ph (e_{\alpha})=\alpha e_{\alpha}$
and $\Ph (e_{\beta})=\beta e_{\beta}.$ If $(p,N)=1$ 
the subspaces $D_{\alpha}=Ee_{\alpha}$ and $D_{\beta}=Ee_{\alpha}$
are regular and the corresponding $p$-adic $L$-functions
$L(V,D_{\alpha},s)$ and $L(V,D_{\beta},s)$ are morally the usual 
$p$-adic $L$-functions associated to $\alpha$ and $\beta$.
If $p\parallel N,$ we have $Ne_{\beta}=e_{\alpha}$ and $Ne_{\alpha}=0$
and $D_{\alpha}$ is the unique regular subspace of $\Dst (V).$ 
The results of Kato \cite{Ka} have the following interpretation 
in terms of our theory. Assume that $(p,N)=1$ and set
$L(V,D_{\alpha},s)=f_{\alpha}(\chi (\g_0)^s-1).$ Then 
$H_{\Iw}^2(V,D_{\alpha})$ is $\CH$-torsion and 
\[
f_{\alpha}\in \mathrm{char}_{\CH}H_{\Iw}^2(V,D_{\alpha})
\]
(see \cite{PoCIT}, Theorem 5.4). In the ordinary case the opposite 
inclusion was recently proved under some
technical conditions by Skinner and Urban \cite{SU}. It would be interesting 
to understand whether those method generalises to the non ordinary case.

5) It is certainly possible to formulate the Main Conjecture 
for families of $p$-adic representations in the spirit of \cite{Gr94b}.

\subsection{The plan of the paper} The first part of the paper is written as a survey article. In Section 1 we review 
the theory of $(\Ph,\Gamma)$-modules. 
In particular, we discuss  recent results of Kedlaya, Pottharst  and
Xiao \cite{KPX} about the cohomology of $(\Ph,\Gamma)$-modules in families
and its applications to the  Iwasawa cohomology. 
In Section 2 we apply this theory to the global Iwasawa theory.
The notion of a regular submodule is discussed 
in Section 2.1. 
In Sections 2.2-2.4 we construct 
the complex $\RG_{\Iw}(V,D)$ and review its basic properties following
Pottharst \cite{PoCIT}. 

The Main conjecture is formulated in Section 2.5.
In the rest of the paper we discuss the relationship between the nullity  
of $H^1_{\Iw}(V,D)$ and the structure of the semistable module $\Dst (V).$
In Section 3 we consider the weight $\leqslant -2$ case and prove that  $H^1_{\Iw}(V,D)=0$  (and therefore  $H^2_{\Iw}(V,D)$ is $\CH$-torsion) 
if the $\ell$-invariant $\ell (V,D)$ constructed in \cite{Ben11}, \cite{Ben14}
does not vanish. In Section 4 we consider the weight $-1$ case and prove 
that the nullity of  $H^1_{\Iw}(V,D)$ follows from the non degeneracy of the 
$p$-adic height pairing. The proof is based on the generalisation 
of Nekov\'a\v r's construction of the $p$-adic  height pairing \cite{NeSC}
to the non ordinary case. A systematic study of $p$-adic heights 
using $(\Ph,\Gamma)$-modules    will be done in  \cite{Ben14b}.

\section{$(\Ph,\Gamma)$-modules}

\subsection{ $(\Ph,\Gamma)$-modules} In this  section we summarize
the results about $(\Ph,\Gamma)$-modules which will be used in 
subsequent Sections.
The notion of a $(\Ph,\Gamma)$-module was introduced by Fontaine in the  fundamental
paper \cite{Fo90}. We consider only $(\Ph,\Gamma)$-modules over $\Qp$ and those families 
 because this is sufficient  for applications we have in mind and refer to original papers for 
the general case.

Let  $p$ be  a prime number.  Fix an algebraic closure  $\overline{ \Q}_p$
of $\Qp$ and set $G_{\Qp}=\Gal (\overline{\Q}_p/\Qp).$ 
Let $\Qp^{\cyc}=\Qp (\mu_{p^{\infty}})$ 
and  $\Gamma =\Gal (\Qp^{\cyc}/\Qp).$ 
The cyclotomic character $\chi \,:\,G_{\Qp} \rightarrow  \Zp^*$ induces an isomorphism
of $\Gamma $ onto $\Zp^*$ which we denote again by $\chi \,:\,\Gamma \rightarrow \Zp^*.$
Let  $\mathbf C_p$ 
 be the  $p$-adic completion of $\overline{\Q}_p.$ We denote by  $\vert \cdot \vert_p$ 
the $p$-adic absolute value on $\mathbf C_p$ normalized by $\vert p\vert_p=1/p.$
We fix a coefficient field $E$ which will be a finite extension of $\Qp$
and consider the following objects:

$\bullet$ The field $\mathscr E_E$ of power series 
$f(X)=\displaystyle\underset{k\in \mathbf Z}\sum a_kX^k,\quad a_k\in E$ such that $a_k$ are $p$-adically bounded and $a_k\to 0$ when $k\to -\infty.$ Thus $\mathscr E_E$ is a complete discrete valuation field with residue field
$k_E((X))$ where $k_E$ is the residue field of $E.$

$\bullet$ For each $0 \leqslant r<1$  the ring $\CR_E^{(r)}$
of  $p$-adic functions 
\begin{equation}
\label{eq:2}
f(X)=\underset{k\in \mathbf Z}\sum a_kX^k,\qquad a_k\in E  
\end{equation}
which are  holomorphic on the $p$-adic annulus
\[
A (r,1)=\{ z\in \mathbf C_p \mid  p^{-1/r}\leqslant \vert z\vert_p  <1\}.
\]
The Robba ring of power series with coefficients in $E$ is defined as 
\[
\CR_E=\underset{r}\cup \CR_E^{(r)}.
\]
Note that $\CR_E$ is a B\'ezout ring
(each finitely generated ideal is principal) \cite{La62} 
but it is not noetherian.
Its group of units coincides with the group of units of 
$\mathscr O_E[[X]]\otimes \Qp$ where $\mathscr O_E$ is the ring of integers of $E.$

$\bullet$  For each $0 \leqslant r<1$ the ring  $\mathscr E_E^{\dagger,r}$ of $p$-adic functions \eqref{eq:2} that
are bounded on $A(r,1).$ Then  $\mathscr E_E^{\dagger}= \underset{r}\cup \mathscr E_E^{\dagger,r}$ is a field  which is contained both in $\mathscr E_E$ and $\CR_E.$
Its elements are called overconvergent power series.

The rings  $\mathscr E_E$, $\mathscr E_E^{\dagger}$ and $\CR_E$
 are equipped  with  an $E$-linear  continuous  action of $\Gamma$  
defined by
\[
g (f(X)) =f((1+X)^{\chi (g)}-1),\qquad g \in \Gamma 
\]
and a linear  operator $\Ph$ called the Frobenius and given by
\[
\Ph (f(X))= f((1+X)^p-1).
\]
Note that the actions of $\Gamma$ and $\Ph$ commute with each other. 
 Set $t=\log (1+X)=\displaystyle \underset{n=1}{\overset \infty \sum} (-1)^{n-1}\frac{X^n}{n}.$ 
Then $t\in \CR$ and   $\gamma (t)=\chi (\gamma)\,t,$ $\Ph (t)=p\,t.$

\begin{definition} i)  A $(\Ph,\Gamma)$-module over $R=\mathscr E_E,\,\,\mathscr E_E^{\dagger}$ or $\CR_E$  is a finitely generated free $R$-module 
$\mathbf D$  equipped
with commuting semilinear actions of $\Gamma$ and $\Ph$ and such that 
$A\Ph (\bD)=\bD$. The last condition  means simply that $\Ph (e_1),\ldots ,\Ph (e_d)$ is a basis
of $\bD$ if $e_1,\ldots ,e_d$ is.
\end{definition}
We denote by $\mathbf M_{A}^{\Ph,\Gamma}$ the category of $(\Ph,\Gamma)$-modules
over $A$.
The category   $\mathbf M_{\mathscr E_E}^{\Ph,\Gamma}$ 
contains the important subcategory
$\mathbf M_{\mathscr E_E}^{\Ph,\Gamma,\text{\rm \'et} }$
of \'etale modules. A $(\Ph,\Gamma)$-module $\bD$ over $\mathscr E_E$ is
\'etale if it is isoclinic of slope $0$ in the sense of Dieudonn\'e-Manin's theory.
More explicitly, $\bD$ is \'etale if there exists a $(\Ph,\Gamma)$-stable lattice 
$\mathscr O_{\mathscr E_E}e_1+\cdots  +\mathscr O_{\mathscr E_E}e_d$ of $\bD$ over the ring
of integers  $\mathscr O_{\mathscr E_E}$ of $\mathscr E_E$ 
such that the matrix of $\Ph$ in the basis $\{e_1,\ldots ,e_d\}$ is invertible
over $\mathscr O_{\mathscr E_E}.$ The category   
$\mathbf M_{\mathscr E_E^{\dagger}}^{\Ph,\Gamma,\text{\rm \'et}}$  of 
\'etale modules over $\mathscr E_E^{\dagger}$ can be defined by the same 
manner.

Let $\Rep_{E}(G_{\Qp})$ denote the $\otimes$-category  of $p$-adic representations  
of the Galois group $G_{\Qp}=\Gal (\overline{\Q}_p/\Qp)$ 
on finite dimensional $E$-vector spaces. Using the field-of-norms functor  of Fontaine-Wintenberger  \cite{W} Fontaine constructed an equivalence of $\otimes$-categories
\[
\bD \,\,:\,\,\Rep_{E}(G_{\Qp})\rightarrow 
\mathbf M_{\mathscr E_E}^{\Ph,\Gamma,\text{\rm \'et} }
\]
and conjectured that each $p$-adic representation $V$  is overconvergent i.e.
$\bD (V)$ has a canonical $\mathscr E_E^{\dagger}$-lattice $\bD^{\dagger} (V)$ 
stable under the actions of $\Ph$ and $\Gamma.$ This was proved by Cherbonnier 
and Colmez in \cite{CC1}.

The relationship between $p$-adic representations
and $(\Ph,\Gamma)$-modules over the Robba ring  can be summarized in the following diagram
\begin{displaymath}
\xymatrix{
\Rep_E(G_{\Qp}) \ar[rr]^{\bD^{\dagger}} 
\ar[ddrr]^{\Ddagrig} &
&\mathbf M_{\mathscr E_E^{\dagger}}^{\Ph,\Gamma,\text{\rm \'et}}
\ar[dd]^{\otimes_{\mathscr E_E^{\dagger}} \CR_E}\\
&  &\\
& &\mathbf M_{\CR_E}^{\Ph,\Gamma}. 
}
\end{displaymath}
A striking fact is that the vertical arrow is a fully faithful functor. 
This follows from Kedlaya's generalization of Dieudonn\'e-Manin theory 
\cite{Ke}. More precisely,the functor  $\bD \mapsto \bD\otimes_{\mathscr E_E^{\dagger}} \CR_E$ establishes
an equivalence between $\mathbf M_{\mathscr E_E^{\dagger}}^{\Ph,\Gamma,\text{\rm \'et}}$
and the category of $(\Ph,\Gamma)$-modules over $\CR_E$ of slope $0$ in the sense
of Kedlaya. See \cite{Cz08}, Proposition 1.7 for details.

\subsection{Cohomology of $(\Ph,\Gamma)$-modules} 
On the categories of $(\Ph,\Gamma)$-modules one can define  
cohomology theories whose formal properties are very similar to 
properties of continuous Galois cohomology of local fields. The main idea is 
particulary clear if we work with the  category of \'etale modules. 
Since $\mathbf M_{\mathscr E_E}^{\Ph,\Gamma,\text{\rm \'et} }$ is equivalent 
to $\Rep_E(G_{\Qp})$ it is possible to compute  continuous
Galois cohomology $H^*(\Qp,V)$ in terms of $\bD (V).$ In practice,
since the category of $p$-adic representations has no enough of injective
objects one should first work modulo $p^m$-torsion and consider the category of inductive limits of $(\Ph,\Gamma)$-modules over  $\mathscr O_{\mathscr E_E}/p^m.$ 
This leads to the following results \cite{H1}, \cite{H2}.
Let $\Delta =\Gal (\Qp(\mu_p)/\Qp).$
Then $\Gamma\simeq \Delta\times \Gamma_0$ where $\Gamma_0$ is a procyclic $p$-group.
Fix a topological  generator $\g_0$ of $\Gamma_0$ and set $\g_n=\g_1^{n}.$ $n\geqslant 0.$
Let  $K_n=\Qp(\mu_{p^n})^{\Delta},$
and 
$K_\infty=\underset{n\geqslant 0}\cup K_n.$ Then $\Gamma_0=\Gal (K_\infty/\Qp).$
If $M$ is a topological module equipped with a continuous action of $\Gamma$ and 
an operator $\Ph$ which commute to each other we denote by $C^{\bullet}_{\Ph,\gamma_n}(M)$
the complex 
\begin{equation}
\label{eq:4}
C^{\bullet}_{\Ph,\gamma_n}(M)\,\,:\,\,   0\rightarrow M^{\Delta} \xrightarrow{d_0} M^{\Delta} \oplus M^{\Delta} 
\xrightarrow{d_1}M^{\Delta} \rightarrow 0 ,\qquad n\geqslant 0
\end{equation}
where $d_0(x)=((\Ph-1)x,(\g_n-1)x)$ and $d_1(y,z)=(\g_n-1)y-(\Ph-1)z$ and define  
\[
H^*(K_n,M)=H^*(C^{\bullet}_{\Ph,\gamma_n}(M)).
\] 

Let $\bD$ be an \'etale  $(\Ph,\Gamma)$-module over $\mathscr  E_E$
and let $V$ be a $p$-adic representation of $G_{\BQ_p}$ such that
$\bD=\bD (V).$ Then there  exist isomorphisms
$$
H^*(K_n,\bD) \simeq H^*(K_n,V)
$$
which are functorial and canonical up to the choice of the generator
$\g\in \Gamma$ (see \cite{H1}). This gives an alternative approach
to the Euler-Poincar\'e characteristic formula and the 
the Poincar\'e duality for Galois cohomology of local fields 
\cite{H1}, \cite{H2}. If $\bD$ is an  an \'etale 
 $(\Ph,\Gamma)$-module over $\mathscr E^{\dagger}_E,$ 
 the theorem of Cherbonnier-Colmez implies
 that again $H^i(K_n,\bD) \simeq H^i(K_n,V)$ where $V$ is
 a $p$-adic representation such that $\bD=\Ddagrig (V).$
The cohomology of $(\Ph,\Gamma)$-modules over $\CR_E$  
was studied in detail in \cite{Li} using a non trivial reduction 
to the slope $0$ case. For any $(\Ph,\Gamma)$-module $\bD$ over
$\CR_E$  we consider  $\bD (\chi)=\bD \otimes_E E(\chi)$ equipped with diagonal actions of $\Gamma$ and $\Ph$ (here $\Ph$ acts trivially on $E (\chi)$).  
The main properties of the cohomology groups $H^i(K_n,\bD)$ of $\bD$ are:

{\textit{ 1)  Long cohomology sequence.}} A short exact sequence of $(\Ph,\Gamma)$-modules  over $\CR_E$
\[
0\rightarrow \bD'\rightarrow \bD\rightarrow \bD''\rightarrow 0
\]
gives rise to an exact  sequence
\begin{multline}
0\rightarrow H^0(K_n,\bD')\rightarrow H^0(K_n,\bD)\rightarrow H^0(K_n,\bD)
\xrightarrow{\delta^0}H^1(K_n,\bD')\rightarrow  \cdots \\
\rightarrow H^2(K_n,\bD'')\rightarrow 0. 
\end{multline}

{\textit{ 2)  Euler-Poincar\'e characteristic.}} Let $\bD$ be a 
$(\Ph,\Gamma)$-module over $\CR_E.$
Then  $H^i(\bD)$ are finite dimensional $E$-vector spaces and
\begin{equation}
\label{Euler-Poincare phi gamma}
\chi (K_n,\bD)=\underset{i=0}{{\overset{2}{\sum}}} (-1)^i \dim_E H^i(K_n,\bD)\,=
\,-[K_n:\Qp]\,\textrm{rank}_{\CR_E}(\bD). 
\end{equation}

{\textit{3) Computation of the Brauer group.}} The map 
\begin{equation}
\cl (x) \mapsto -\frac{p^n}{\log \chi (\gamma_n)}\res (xdt) 
\end{equation}
is well defined and induces  an isomorphism 
$\textrm{inv}_{K_n}\,:\,H^2(K_n, \CR_E(\chi)) \iso E$.

{\textit{4) Cup-products.}} Let $\bD'$ and $\bD''$ be two $(\Ph,\Gamma)$-modules
over $\CR_E$. 
For all $i$ and $j$ such that $i+j\leqslant 2$ define a bilinear map
\[
\cup \,:\, H^i(K_n,\bD') \times H^j(K_n,\bD'') \rightarrow  H^{i+j} (K_n,\bD'\otimes \bD'')
\] 
by 
\begin{eqnarray}
\label{cup product formula}
\cl (x) \cup \cl (y) =\cl (x\otimes y)\qquad \text{if $i=j=0$},\nonumber\\
\cl (x) \cup \cl (y_1,y_2)=\cl (x\otimes y_1, x\otimes y_2) \qquad \text{if $i=0$, $j=1$},
\nonumber\\
\cl (x_1,x_2)\cup \cl (y_1,y_2)=\cl (x_2\otimes \gamma_n (y_1)-x_1\otimes \Ph (y_2)) \qquad \text{if $i=1$, $j=1$}, \nonumber\\
\cl (x)\cup \cl (y)=\cl (x\otimes y) \qquad \text{if $i=0$, $j=2$}.
\end{eqnarray}

These maps commute with  connecting homomorphisms in the usual sense.

{\textit{5) Duality.}}  Let $\bD^*=\textrm{Hom}_{\mathscr \CR_E} (\bD, \CR_E).$ For $i=0,1,2$ the cup product
\begin{equation}
\label{duality for coh phi-gamma modules}
H^i(K_n,\bD)\times H^{2-i}(K_n,\bD^*(\chi)) \xrightarrow{\cup} H^2(K_n, \CR_E(\chi))\simeq E 
\end{equation}
is a perfect pairing.

{\it 6) Comparision with Galois cohomology.} Let $\bD$ be an etale $(\Ph,\Gamma)$-
module over $\mathscr E_E^{\dagger}$. Then the  map
\[
C^{\bullet}_{\Ph,\g_n}(\bD)\rightarrow C^{\bullet}_{\Ph,\g_n}(\bD\otimes \CR_E)
\]
induced by the natural map $\bD\rightarrow \bD\otimes_{\mathscr E_E^{\dagger}} \CR_E$
is a quasi-isomorphism. In particular, if $V$ is a $p$-adic representation
of $G_{\Qp}$ then 
\[
H^*(K_n,V) \iso H^*(K_n,\Ddagrig (V)).
\]

\subsection{Relation to the $p$-adic Hodge theory} In \cite{Fo90}, Fontaine proposed to classify 
$p$-adic representations
arising in the $p$-adic Hodge theory in terms of $(\Ph,\Gamma)$-modules (Fontaine's
program). More precisely, the problem is  to recover classical Fontaine's functors
$\Dd (V),$ $\Dst (V)$ and $\Dc (V)$ (see for example \cite{Fo94b}) from $\Ddagrig (V).$ The compete solution was
obtained by Berger in \cite{Ber02}, \cite{Ber08}. His theory also  allows 
to prove that each de Rham representation is potentially semistable. 
See also   \cite{Cz03} for introduction and relationship to the
theory of $p$-adic differential equations. In this section we 
review some of  results of Berger in the case where  the ground field 
is $\BQ_p.$ 
Consider the following categories 

$\bullet$ The category $\mathbf M\mathbf F_E$ of finite dimensional $E$ vector spaces $M$ equipped
with an exhaustive decreasing filtration $(\F^iM)_{i\in\mathbf Z}.$

$\bullet$ The category  $\mathbf M\mathbf F^{\Ph,N}_E$  
finite dimensional $E$ vector spaces $M$ equipped
with an exhaustive decreasing filtration $(\F^iM)_{i\in\mathbf Z},$ 
a linear  bijective Frobenius  
map $\Ph \,:\,M\rightarrow M$   and a nilpotent operator (monodromy) 
$N\,:\,M\rightarrow M$ such that $\Ph N=p\,\Ph N.$

$\bullet$ The subcategory $\mathbf M\mathbf F^{\Ph}_E$ of $\mathbf M\mathbf F^{\Ph,N}_E$
formed by filtered $(\Ph,N)$-modules $M$ such that $N=0$ on $M$.

Let  $\Qp^{\cyc}((t))$
be the ring of Laurent power series 
equipped with the filtration 
$\F^i\Qp^{\cyc}((t))=t^i\Qp^{\cyc}[[t]]$ and the action of $\Gamma$ given by
$g \left (\displaystyle\underset{k\in\mathbf Z}\sum a_kt^k \right )=
\displaystyle\underset{k\in\mathbf Z}\sum g(a_k)\chi (g)^kt^k.$ 
The ring $\CR_{E}$ can not be naturally embedded in 
\linebreak
$E\otimes\Qp^{\cyc}((t))$
but for any $r>0$ small enough and $n\gg 0$  there exists a $\Gamma$-equivariant  embedding 
$i_n\,:\,\CR_{E}^{(r)}\rightarrow E\otimes \Qp^{\cyc}((t))$ which sends  $X$
onto $\zeta_{p^n}e^{t/p^n}-1.$ 
Let $\bD$ be a $(\Ph,\Gamma)$-module over $\CR_E.$ One can construct for such $r$ 
a natural $\Gamma$-invariant  $\CR^{(r)}_E$-lattice $\bD^{(r)}.$ Then
\[
\mathscr D_{\dR} (\bD)=\left ( E\otimes \Qp^{\cyc}((t))\otimes_{i_n}\bD^{(r)} \right )^{\Gamma}
\]
is a finite dimentional $E$-vector space equipped with a decreasing filtration
\[
\F^i\mathscr \Dd (\bD)=\left ( E\otimes \F^i\Qp^{\cyc}((t))\otimes_{i_n}\bD^{(r)} \right )^{\Gamma}
\]
which does not depend on the choice of $r$ and  $n.$ 

Let $\CR_E[\ell_X]$ denote the ring of power series with coefficients 
in $\CR_E.$ Extend the actions of  $\Ph$ and $\Gamma$ to $\CR_E[\ell_X]$ setting
\[
\Ph (\ell_X) =p\ell_X+ \log \left (\frac{\Ph (X)}{X^p} \right ),
\qquad
g (\ell_X) =\ell_X + \log \left (\frac{g(X)}{X} \right ),\quad g\in \Gamma
\]
( note that $\log (\Ph (X)/X^p)$ and 
$\log ( g (X)/X)$  converge in $\CR_E$). Define a monodromy operator 
$
N\,:\,\CR_E[\ell_X]\rightarrow \CR_E[\ell_X]
$
by  $N = -\displaystyle\left (1-\frac{1}{p} \right )^{-1}\displaystyle \frac{d}{d \log X}.$
For any $(\Ph,\Gamma)$-module $\bD$ define 
\begin{eqnarray}
\CDst (\bD)= \left (\bD\otimes_{\CR_E} \CR_{E}[\ell_X, 1/t  ] \right )^{\Gamma}, \quad t=\log (1+X),\\
\CDcris (\bD)= \CDst (\bD)^{N=0}= (\bD [1/t])^{\Gamma}.
\end{eqnarray}
Then $\CDst (\bD)$ is a finite dimensional  $E$-vector space equipped with  natural actions of $\Ph$ and $N$ such that $N\Ph=p\,\Ph N.$ Moreover, it is equipped with a canonical exhaustive decreasing filtration induced by the embeddings $i_n.$ We have therefore 
three functors
\[
\mathscr D_{\dR} \,:\,{\mathbf M}^{\Ph,\Gamma}_{\CR_E} \rightarrow \mathbf M\mathbf F_{E},
\quad
\mathscr D_{\st} \,:\,{\mathbf M}^{\Ph,\Gamma}_{\CR_E} \rightarrow
\mathbf M\mathbf F^{\Ph,N}_{E},\quad 
\mathscr D_{\cris} \,:\,{\mathbf M}^{\Ph,\Gamma}_{\CR_E} \rightarrow 
\mathbf M\mathbf F^{\Ph}_{E}.
\]

\begin{mytheorem}[\text{\sc Berger}]
\label{berger theorem}
 Let $V$ be a $p$-adic representation of $G_{\Qp}.$
Then 
\[
\bD_*(V)\simeq \mathscr D_*(V),\qquad *\in\{\text{\rm dR},\text{\rm st},\text{\rm cris}\}.
\]
\end{mytheorem}
\begin{proof} See \cite{Ber02}.
\end{proof}

For any $(\Ph,\Gamma)$-module over $\CR_E$ one has 
\[
\dim_E \CDcris (\bD)\leqslant \dim_ E\CDst (\bD) \leqslant
\dim_E\mathscr D_{\dR}(\bD)\leqslant 
 \textrm{rg}_{\CR_E} (\bD).
\] 
One says that $\bD$ is de Rham (resp. semistable, resp. crystalline) if 
\linebreak
$\dim_L \mathscr D_{\dR} (\bD)= \textrm{rg}_{\CR_E} (\bD)$
(resp. $\dim_ E\CDst (\bD) = \textrm{rg}_{\CR_E} (\bD)$, resp.
\linebreak
$\dim_E \CDcris (\bD)= \text{rg}_{\CR_E} (\bD)$).
Let
$\mathbf M^{\Ph,\Gamma}_{\CR_E,\textrm{pst}}$ and $\mathbf M^{\Ph,\Gamma}_{\CR_E,\cris}$
denote  the categories
of semistable and crystalline 
$(\Ph,\Gamma)$-modules  respectively.
Berger proved  (see \cite{Ber08}  ) that the functors 
\begin{equation}
\label{equivalence Dst and Dcris for phi-gamma mod}
\CDst \,:\,\mathbf M^{\Ph,\Gamma}_{\CR_E,\text{st}} \rightarrow \mathbf M\mathbf F^{\Ph,N}_{E},\qquad
\CDcris \,:\,\mathbf M^{\Ph,\Gamma}_{\mathscr R_E,\cris}\rightarrow \mathbf M\mathbf F^{\Ph}_{E}  
\end{equation}
are equivalences of $\otimes$-categories.
If $\bD$ is de Rham, the jumps of the filtration 
$\F^i\mathscr D_{\dR} (\bD)$ will be  called the  Hodge-Tate weights of $\bD.$ 

\subsection{Families of $(\Ph,\Gamma)$-modules} In this section we review the theory of $(\Ph,\Gamma)$-modules 
in families  and its relationship to families of $p$-adic representations
following \cite{BCz}, \cite{KL10}, \cite{Po13} and \cite{KPX}. 
Let $A$ be an affinoid algebra over $\Qp.$ For each $0\leqslant r<1$ the ring $\CR^{(r)}_{\Qp}$ 
is equipped with a canonical  Fr\'echet topology (see \cite{Ber02}) and we define
$\CR_{A}^{(r)}=A\widehat\otimes_{\Qp}  \CR_{\Qp}^{(r)}.$ Set
$\CR_{A}=\underset{0\leqslant r<1}\cup \CR_{A}^{(r)}.$ The actions
of $\Ph$ and $\Gamma$ on $\CR_{\Qp}$ extend by linearity to $\CR_{A}.$ 
We remark that $\CR_A^{(r)}$ is stable under the action of $\Gamma$ 
and that $\Ph (\CR_A^{(r)})\subset \CR_A^{(r^{1/p})}$
( but $\Ph (\CR_A^{(r)})\not \subset \CR_A^{(r)}$).
In order to obtain objects with reasonable behavior one defines
first $(\Ph,\Gamma)$-modules over $\CR_A^{(r)}.$ 
$(\Ph,\Gamma)$-modules over $\CR_A$ are defined by extension of coefficients
from $\CR_A^{(r)}$ to $\CR_A.$

\begin{definition} i) A $(\Ph,\Gamma)$-module over $\CR_A^{(r)}$ is a finite 
projective $\CR_A^{(r)}$ -module $\bD^{(r)}$ equipped with the following structures

a) An isomorphism of $\CR_A^{(r^{1/p})}$-modules 
\[
\Ph^*\,:\,\bD^{(r)}\otimes_{\CR_A^{(r)},\Ph}\CR_A^{(r^{1/p})} 
\iso \bD^{(r^{1/p})}.
\]

b) A semilinear continuous action of $\Gamma$ on $\bD^{(r)}.$ 

ii) One says that $\bD$ is a $(\Ph,\Gamma)$-modules over $\CR_A$ if  
$\bD=\bD^{(r)}\otimes_{\CR_A^{(r)}}\CR_A$ for some $(\Ph,\Gamma)$-module $\bD^{(r)}$
over $\CR_A^{(r)}.$
\end{definition}

The following proposition shows that if $A=E$ this definition is compatible
with the definition given in Section 1.1.

\begin{myproposition} Let $\bD$ be a $(\Ph,\Gamma)$-module over $\CR_E.$
There exists $0\leq r(\bD)<1$ such that for any $r$ such that $r(\bD)\leq r<1$ 
there exists a unique free $\CR_E^{(r)}$-submodule $\bD^{(r)}$ of $\bD$ 
having the following properties

a) $\bD^{(s)}=\bD^{(r)}\otimes_{\CR_E^{(r)}}\CR_E^{(s)}$ for  $r\leqslant s<1.$

b)  $\bD =\bD^{(r)}\otimes_{\CR_E^{(r)}}\CR_E^.$

c) The Frobenius $\Ph$ induces isomorphisms
$$
\Ph^*\,:\,\bD^{(r)}\otimes_{\CR_E^{(r)},\Ph}\CR_E^{(r^{1/p})} 
\iso \bD^{(r^{1/p})},\quad r(\bD)\leqslant r<1.
$$
\end{myproposition}
\begin{proof} This is Theorem 1.3.3 of \cite{Ber08}.
\end{proof}

Let $\Rep_A(G_{\Qp})$ be the category of  projective $A$-modules
of finite rank equipped with an $A$-linear continuous action of $G_{\Qp}.$
The construction of the functor  $\Ddagrig$ can be directly generalized to this case 
\cite{BCz}, \cite{KL10}. More precisely,  there exists a  fully faithful exact functor
\[
\DdagrigA\,:\,\Rep_A(G_{\Qp})\rightarrow \mathbf M^{\Ph,\Gamma}_{\CR_A}
\]
which commutes with base change. Let  $\mathscr X=\mathrm{Spm}(A).$
For each  $x\in \mathscr X$ we denote by $\mathfrak m_x$ the maximal ideal 
of $A$ associated to $x$ and $E_x=A/\mathfrak m_x.$ If $V$ (resp. $\bD$)
is an object of $\Rep_A(G_{\Qp})$ (resp. of 
$\mathbf M^{\Ph,\Gamma}_{\CR_A}$) we set $V_x=V\otimes_AE_x$
(resp. $\bD_x=\bD\otimes_AE_x$). 
Then the diagram
\[
\xymatrix
{\Rep_A(G_{\Qp}) \ar[r]^-{\DdagrigA} \ar[d]^{\otimes E_x} &\mathbf M^{\Ph,\Gamma}_{\CR_A}
\ar[d]^{\otimes E_x}
\\
\Rep_{E_x}(G_{\Qp}) \ar[r]^-{\Ddagrig}  &\mathbf M^{\Ph,\Gamma}_{\CR_{E_x}}
}
\]
commutes 
i.e. 
$
\DdagrigA (V)_x\simeq \Ddagrig (V_x).
$
We remark that in general  the essential image of $\DdagrigA$ does not
coincide with the subcategory of \'etale modules. See \cite{BCz}
\cite{KPX}, \cite{Hel12} for further discussion. 

Let $\bD$ be a $(\Ph,\Gamma)$-module over $\CR_A.$ As in the case $A=E$ we 
attach  to $\bD$  Fontaine-Herr complexes $C_{\Ph,\gamma_n}^{\bullet}(\bD)$
and consider the associated cohomology groups $H^*(K_n,\bD).$ 
We summarize 
the main properties of these cohomology  in the theorem below. 
The key  result here  is the finiteness of the rank of $H^*(K_n,\bD).$

\begin{mytheorem}
\label{cohomology in families}
 Let $A$ be an affinoid algebra over $\Qp$  and let
$\bD$ be a $(\Ph,\Gamma)$-module over $\CR_A.$ Then

i) Finiteness.  The cohomology groups $H^i(K_n,\bD)$ are finitely generated $A$-modules. More precisely,
for each $n\geqslant 0$ 
the complex $C_{\Ph,\g_n}^{\bullet}(\bD)$ is 
quasi-isomorphic to the complex of projective $A$-modules of finite rank concentrated
in degrees $0,$ $1$ and $2$. 

ii) Base change. If $f\,:\,A\rightarrow B$ is  a morphism of affinoid algebras, then
\[
C^{\bullet}_{\Ph,\g_n}(\bD)\otimes ^{\mathbf L}_{\CR_A}\CR_B\iso C^{\bullet}_{\Ph,\g_n}(\bD\widehat\otimes_{\CR_A}\CR_B).
\]
In particular, if $x\in \mathscr X,$ then 
\[
C^{\bullet}_{\Ph,\g_n}(\bD)\otimes ^{\mathbf L}_{\CR_A}E_x\iso C^{\bullet}_{\Ph,\g_n}(\bD_x).
\]

iii) Euler-Poincar\'e formula. One has 
\[
\chi (K_n,\bD)= \underset{i=0}{\overset{2}{\sum}} \textrm{rank}_A H^i(K_n,\bD)\,=
\,-[K_n:\Qp]\,\mathrm{rg}_{\CR_A} (\bD)
\]
where the rank is considered as a function $\text{\rm rg}_A\,:\,\text{\rm Spm}(A)\rightarrow \mathbf N.$

iv) Duality.  The formulas  \eqref{cup product formula} define a duality
{\rm
\[
C^{\bullet}_{\Ph,\g_n}(\bD) \iso \R\textrm{Hom}_A (C^{\bullet}_{\Ph,\g_n}(\bD^*(\chi)),A)[-2].
\]
}
v) Comparision with Galois cohomology.  Let $V$ be a $p$-adic representation with coefficients in
$A$. Then there are functorial  isomorphisms
{\rm
\[
H^*(K_n,V)\iso H^*(K_n, \DdagrigA (V))
\]
}
\end{mytheorem} 
\begin{proof} See \cite{KPX}, Theorems 4.4.1, 4.4.2, 4.4.5 and \cite{Po13}, Theorem 2.8.
\end{proof}

\subsection{ Derived categories} There exists the derived version of the
comparision isomorphisms v) of Theorem \ref{cohomology in families}
which is  important for the formalism of Iwasawa theory \cite{Ben14}, \cite{Po13}. Let $A$ be an affinoid algebra over $\Qp.$
For any $(\Ph,\Gamma)$-module $\bD$ over $\CR_A$ define
\[
C_{\g_n}^{\bullet}(\bD) =\left [\bD \xrightarrow{\gamma_n-1} \bD \right ]
\]
where the first term is placed in degree $0.$ Then $C_{\Ph,\g_n}^{\bullet}(\bD)$
can be defined as the total complex 
\[
C_{\Ph,\g_n}^{\bullet}(\bD)=\Tot^{\bullet} \left (C_{\g_n}^{\bullet}(\bD)
\xrightarrow{\Ph-1} C_{\g_n}^{\bullet}(\bD) \right ).
\]

If $\bD'$ and $\bD''$ are two $(\Ph,\Gamma)$-modules over $\CR_A$, the cup product 
\[
\cup_{\g}\,\,:\,\,C_{\g_n}^{\bullet}(\bD') \otimes C_{\g_n}^{\bullet}(\bD'')
\rightarrow C_{\g_n}^{\bullet}(\bD'\otimes \bD'')
\]
defined by
\[
\cup_{\g}(x_i\otimes y_j)=
\begin{cases} x_i\otimes \g^n(y_j) &\text{if $x_i\in C_{\g_n}^i(\bD'),$ $y_j\in C_{\g_n}^j(\bD'')$ and 
$i+j=0$ or $1$},\\
0 &\text{if $i+j\geqslant 2.$}
\end{cases}
\] 
gives rise to a map of complexes
\[
\cup \,:\,C^{\bullet}_{\Ph,\g_n}(\bD') \otimes C^{\bullet}_{\Ph,\g_n} (\bD'')
\rightarrow  C^{\bullet}_{\Ph,\g_n}(\bD'\otimes \bD'').
\]
Explicitly 
\begin{equation}
\cup ((x_{i-1},x_i)\otimes (y_{j-1},y_j))=
(x_i\cup_{\g}y_{j-1}+(-1)^j x_{i-1}\cup_{\g}\Ph (y_j), x_i\cup_{\g}y_j)
\end{equation}
if $(x_{i-1},x_{i})\in C^{i}_{\Ph,\g_n}(\bD')=C_{\g_n}^{i-1}(\bD')\oplus C_{\g_n}^i(\bD')$ and
$(y_{j-1},y_{j})\in C^{j}_{\Ph,\g_n}(\bD'')=C_{\g_n}^{j-1}(\bD'')\oplus C_{\g_n}^j(\bD'').$
It is easy to see that this is a bilinear map which induces the cup-product 
\eqref{duality for coh phi-gamma modules}
on cohomology groups.

Let $H_{\BQ_p}=\Gal (\overline\BQ_p/\Qp^{\cyc}).$ In \cite{Ber02}, Berger constructed a topological ring $\BrdAr$
equipped with commuting actions of $\Ph$ and $G_{\BQ_p}$ such that
$\bD^{\dagger, r}_{\mathrm{rig},A}(V)\subset V\otimes_A \BrdAr$     
for each Galois representation $V.$ Moreover one has an exact sequence 
\begin{equation}
\label{exact sequence for BrdAr}
0\rightarrow A\rightarrow \BrdAr \xrightarrow{\Ph-1}\BrdArp \rightarrow 0
\end{equation} 
(see \cite{Ber03}, Lemma 1.7). 
Set $\BrdA=\underset{0<r}\varinjlim \BrdAr.$ Passing to
limits in \eqref{exact sequence for BrdAr} we obtain
an exact sequence 
\begin{equation}
\label{exact sequence for BrdA}
0\rightarrow A\rightarrow \BrdA \xrightarrow{\Ph-1}\BrdA \rightarrow 0
\end{equation} 
Set  $G_{K_n}=\Gal (\overline{\Q}_p/\Qp (\mu_{p^n})).$  Tensoring  \eqref{exact sequence for BrdA}
with a $p$-adic representation $V$ and taking continuous cochains one obtains an exact sequence
\[
0\rightarrow C^{\bullet}(G_{K_n},V)\rightarrow 
C^{\bullet}\left (G_{K_n},V\otimes \BrdA \right )\xrightarrow{\Ph-1}
C^{\bullet}\left (G_{K_n},V\otimes \BrdA \right )\rightarrow 0.
\]
Define 
\begin{equation}
\label{definition of K(V)}
K^{\bullet}(K_n,V)=
{\Tot}^{\bullet}\left (C^{\bullet}\left (G_{K_n},V\otimes \BrdA \right )\xrightarrow{\Ph-1}
C^{\bullet}\left (G_{K_n},V\otimes \BrdA \right )
\right ).
\end{equation}
Consider the diagram 
\begin{equation}
\label{diagram with K(V)}
\xymatrix{
C^{\bullet}(G_{K_n},V)\ar[r]^{\beta_V} &K^{\bullet}(K_n,V)\\
& C^{\bullet}_{\Ph,\g_n}(V) \ar[u]^{\alpha_V}}
\end{equation}
where the maps $\alpha_V$ and $\beta_V$  are constructed as follows. 
Let 
$
\alpha_{V,\g}\,:\,C^{\bullet}_{\g_n}(V)\rightarrow  C^{\bullet}(G_{\Qp},V\otimes \BrdA)
$
be the morpism defined by 
\begin{eqnarray}
\alpha_{V,\g} (x_0)=x_0,\qquad\text{\rm if $ x_0\in C^0_{\g_n}(V)$},\nonumber\\ 
\alpha_{V,\g} (x_1) (g)=\frac{g-1}{\g_n-1} (x_1), \qquad \text{\rm if 
$x_1\in C^1_{\g_n}(V)$}\nonumber.
\end{eqnarray}

Then $\alpha_V$ is the map induced by $\alpha_{V,\g}$ by passing
to total complexes.
The map $\beta_V$ is defined by
\begin{eqnarray}
&&\beta_V\,:\,C^{\bullet}(G_{K_n},V)\rightarrow K^{\bullet}(K_n,V),\nonumber\\
&&x_n\mapsto (0,x_n),\qquad x_n\in C^n(G_{K_n},V).\nonumber
\end{eqnarray}

Let $\RG (K_n,V)$ and $\RG (K_n,\Ddagrig (V))$ denote the images of 
complexes 
\linebreak
$C^{\bullet}(G_{K_n},V)$ and $C^{\bullet}_{\Ph,\g_n}(\Ddagrig (V))$ 
in the derived category $\mathrm D^{\mathrm b}(A)$ of $A$-modules. 

\begin{myproposition}
i) In the diagram \eqref{diagram with K(V)} the maps
$\alpha_V$ and $\beta_V$ are quasi isomorphisms and therefore 
\begin{equation}
\label{derived version for comparision iso}
\RG (K_n,V) \iso \RG (K_n,\Ddagrig (V))
\end{equation}
in $\mathrm D^{\mathrm b}(A).$
\end{myproposition}
\begin{proof} See \cite{Ben14}, Proposition A.1.
\end{proof}

\subsection{Iwasawa cohomology}
Let $\CO_E$ denote the ring of integers of the field of 
coefficients $E.$ We 
 equip the Iwasawa algebra $\Lambda=\CO_E[[\Gamma_1]]$ with the 
involution $\iota \,:\,\Lambda \rightarrow \Lambda$ defined by $\iota (g)=g^{-1}$, $g\in \Gamma_1.$
Let $V$ be a $p$-adic representation of $G_{\Qp}$ with coefficients in $E.$
 Fix a $\CO_E$-lattice 
$T$ of $V$ stable under the action of $G_{\Qp}.$
The induced module $\Ind_{K_\infty/\Qp}(T)=\Lambda \otimes_{\CO_E} T$
is equipped with the diagonal action of $G_{\Qp}$ and the natural structure
of a  $\Lambda$-module given by $\lambda *m=\iota (\lambda)m$ (see for example
\cite{NeSC}, Chapter 8. 
To fix these structures we will write  $\Ind_{K_\infty/\Qp}(T)=(\Lambda \otimes_{\Zp} T)^{\iota}.$
Let $\RG_{\Iw}(\Qp,T)$ denote the class of the complex 
$C^\bullet (G_{\Qp},\Ind_{K_\infty/\Qp}(T))$ in the derived category 
$D^{\mathrm b}(\Lambda).$ 
Define
\[
H^i_{\Iw}(\Qp,T)=\R^i\Gamma_{\Iw}(\Qp,T),\qquad i\in \mathbf N.
\]
From Shapiro's lemma it follows that there are canonical and functorial isomorphisms in
$\mathrm D^{\mathrm b}(\Zp[G_n])$ where $G_n=\Gal (K_n/\Qp)$ 
\begin{eqnarray}
&&H^i_{\Iw}(\Qp,T)\simeq\underset{\text{\rm cores}_{K_n/K_{n-1}}}\varprojlim H^i(K_n,T),\nonumber\\
&&\RG_{\Iw}(\Qp,T)\otimes^{\mathbf L}_{\Lambda}\Zp[G_n]\simeq
\RG (K_n,T). \nonumber 
\end{eqnarray}

We review the computation of Iwasawa cohomology in terms of $(\Ph,\Gamma)$-modules.
It was found by Fontaine (unpublished but see \cite{CC2}). Let
\[
\psi \,:\,\mathscr E^{\dagger}_E \rightarrow \mathscr E^{\dagger}_E
\]
denote the operator
\[
\psi (f(X))=\frac{1}{p}\text{\rm Tr}_{\mathscr E^{\dagger}_E/\Ph(\mathscr E^{\dagger}_E)}(f(X)).
\]
More explicitly, the polynomials 
$1$, $(1+X),\ldots ,(1+X)^{p-1}$ form a basis of $\mathscr E^{\dagger}_E$ over $\Ph(\mathscr E^{\dagger}_E)$
and one has 
\[
\psi \left (\underset{i=0}{\overset{p-1}\sum} \Ph (f_i) (1+X)^i\right )=f_0.
\]
In particular, $\psi\circ \Ph=\text{\rm id}.$
Let $e_1,\ldots ,e_d$ be a base of $\bD^{\dagger}(V)$ over $\mathscr E^{\dagger}_E.$
Then  $\Ph (e_1),\ldots ,\Ph (e_d)$ is again a base of $\bD^{\dagger}(V)$ and we define
\[
\psi \,:\,\bD^{\dagger}(V)\rightarrow \bD^{\dagger}(V)
\]
by $\psi\left (\underset{i=1}{\overset{d}\sum} a_i\Ph(e_i)\right )=
\underset{i=1}{\overset{d}\sum} \psi(a_i)e_i.$ We remark that this
definition extends to $(\Ph,\Gamma)$-modules over the Robba ring.

Consider the complex
\[
C^{\bullet}_{\Iw}(\bD^{\dagger}(T))\,:\, \bD^{\dagger}(T)^{\Delta} 
\xrightarrow{\psi-1}\bD^{\dagger}(T)^{\Delta}
\]
where the first term is placed in degree $0$ and denote by
$\RG_{\Iw}(\Ddagrig (T))$ the corresponding object in 
$\mathrm D^{\mathrm b}(\Lambda).$ 

\begin{mytheorem}  The complexes $\RG_{\Iw}(\Qp,T),$ 
$\RG_{\Iw}(\Ddagrig (T))$
and 
\linebreak
$\RG (\Qp,\bD^{\dagger}(\Ind_{K_\infty/\Qp}T))$
are isomorphic in $\mathrm D^{\mathrm b}(\Lambda).$
\end{mytheorem}
\begin{proof} See \cite{CC2} and \cite{Ben14} for the derived version.
\end{proof}

The analog of previous results for $(\Ph,\Gamma)$-modules over the Robba ring
was obtained by Pottharst \cite{PoCIT}. 
We start with some preliminary results about coadmissible $\CH$-modules 
and the Grothedieck duality. 
 Let $B(0,1)=\{z\in \mathbf C_p \mid \vert z\vert_p<1\}$
denote the open unit disc.
 Define 
\[
\CH=\left \{f(X)\in E [[X]] \mid \text{\rm $f(X)$ converges 
on $B(0,1)$}\right \}.
\]
In the context of Iwasawa theory the ring $\CH$  appeared in \cite{PR94}.
The natural description $\CH$ is as follows. Write
\[
B(0,1)=\underset{n}\cup W_n
\]
where  $W_n=\{z\in \mathbb C_p\mid \vert z\vert_p\leqslant p^{-1/n}\}$
is the  closed disc of radii $p^{-1/n}<1$ which we consider as an affinoid space. 
Then
\[
\Gamma (W_n,\mathscr O_{W_n})= 
\left \{f(X)=\underset{k=0}{\overset{\infty}\sum}a_kX^k
\mid \text{\rm $\vert a_k\vert_p p^{-k/n}\to 0$ when $k\to +\infty$}\right \}
\]
and $\CH=\underset{n}\varprojlim \Gamma (W_n,\mathscr O_{W_n}).$
We consider $\Lambda$ and $\Lambda \otimes_{\CO_E}E$
as subalgebras of $\CH.$  To simplify notation set $\CH_n=\Gamma (W_n,\mathscr O_{W_n}).$ It is easy to see that $\CH_n$ are euclidian rings. 
A {\it coadmissible} $\CH$-module $M$ 
is the inverse limit of a system $(M_n)_{n\geqslant 1}$ where each
$M_n$ is a finitely generated $\CH_n$-module and the natural maps
$M_n\otimes_{\CH_n}\CH_{n-1} \rightarrow M_{n-1}$ are  isomorphisms 
\cite{ST}. The structure of admissible modules is given by the following 
proposition (\cite{PoCIT}, Proposition 1.1).

\begin{myproposition}
\label{properties coadmissible modules}
 i)  A coadmissible $\CH$-module is torsion free if and only if it is a finitely generated $\CH$-module.

ii) Let $M$ be a coadmissible torsion $\CH$-module. Then 
\begin{equation}
M\simeq \underset{i\in I}\prod \CH/\mathfrak p_i^{n_i}
\end{equation}
where $(\mathfrak p_i)_{i\in I}\subset \underset{n\geqslant 1}\cup\text{\rm Spec}(\CH_n)$ is a system of maximal ideals 
such that for each $n$ there are only finitely many $i$ with 
$\mathfrak p_i\in \text{\rm Spec}(\CH_n).$
\end{myproposition}
\begin{proof} The proof follows easily from the theory of Lazard 
\cite{La62}.
\end{proof}

Let $D^{\mathrm b}_{\mathrm{coad}}(\CH)$ denote the category of 
bounded complexes of $\CH$-modules with coadmissible cohomology. 
Let $\mathscr K$ de note the field of fractions of $\CH.$ 
Consider the  complex 
\[
\omega=\mathrm{cone} \left [\mathscr K\rightarrow \mathscr K/\CH \right ] [-1]
\]
and for any object $C^{\bullet}$ of $D^{\mathrm b}_{\mathrm{coad}}(\CH)$ define 
\[
\mathscr D (C^{\bullet})=\Hom_{\CH} (C^{\bullet}, \omega).
\]
Then $\mathscr D\,:\,D^{\mathrm b}_{\mathrm{coad}}(\CH) \rightarrow 
D^{\mathrm b}_{\mathrm{coad}}(\CH)$ is an anti involution which can be seen 
as the "limit" of Grothendieck dualizing  functors 
$\R\Hom (-,\CH_n).$ For any coadmissible module $M$ let $\mathscr D^k(M)$ denote the $k$th cohomology group of  $\mathscr D([M]).$  Then 
\begin{eqnarray}
&&\mathscr D^0(M)=\Hom_{\CH}(M/M_{\mathrm{tor}}, \CH),
\qquad \mathscr D^1 \left(\underset{i\in I}\prod \CH/\mathfrak p_i^{n_i}\right )=
\underset{i\in I}\prod \,\mathfrak p_i^{-n_i}/\CH,\nonumber
\end{eqnarray}
and $\mathscr D^k(M)=0$ for all $k \geqslant 2$ ( see \cite{PoCIT}, Section 1).

Let $\bD$ be a $(\Ph,\Gamma)$-module over $\CR_E.$ Consider the complexes
\[
C^{\bullet}_{\Iw} (\bD)\,:\,\bD^{\Delta}\xrightarrow{\psi-1}\bD^{\Delta}
\]
and $C^{\bullet}_{\Ph,\g}(\overline\bD)$ where $\overline\bD=
\bD \otimes_{\Qp}\CH^{\iota}$ and denote by $\RG_{\Iw}(\bD)$
and $\RG (\Qp, \overline\bD)$ the corresponding objects of the derived category
$D^{\mathrm b}_{\mathrm{coad}}(\CH).$

\begin{mytheorem} 
\label{Iwasawa cohomology of phi-gamma mod}
Let $\bD$ be a $(\Ph,\Gamma)$-module over $\CR_E.$ Then 

i) $H^i_{\Iw}(\bD)$ ($i=1,2$) are coadmissible $\CH$-modules.
Moreover $H^1_{\Iw}(\bD)_{\text{\rm tor}}$ and 
$H^2_{\Iw}(\bD)$ are finite dimensional $E$-vector spaces. 

ii) The complexes $C^{\bullet}_{\Iw}(\bD)$ and $C^{\bullet}_{\Ph,\g}(\overline\bD)$ 
are quasi isomorphic and therefore 
\[
\RG_{\Iw}(\bD) \simeq  \RG (\Qp, \overline\bD).
\]

iii) One has an isomorphism 
\[ 
C^{\bullet}_{\Ph,\g}(\overline\bD)\otimes_{\CH}E \iso C^{\bullet}_{\Ph,\g}(\bD)
\]
which induces Hochschild-Serre exact sequences
\[
0\rightarrow H^i_{\Iw}(\bD)_{\Gamma} \rightarrow 
H^i(\bD) \rightarrow 
H^{i+1}_{\Iw}(\bD)^{\Gamma}\rightarrow 0.
\]

iv) One has a canonical duality in $D^{\mathrm b}_{\mathrm{coad}}(\CH)$
\[
\mathscr D\RG_{\Iw} (\bD) \simeq \RG_{Iw}(\bD^*(\chi))^{\iota}[2].
\]

iv) Let $V$ be a $p$-adic representation of $G_{\Qp}.$  There exist 
canonical and functorial  isomorphisms 
\[
\RG_{\Iw}(\Qp,T)\otimes_{\Lambda}^{\mathbf L} \CH \simeq
\RG (\Qp,T\otimes_{\Zp}\CH^{\iota})
\simeq 
\RG (\Qp, \overline\bD).
\]
\end{mytheorem}
\begin{proof}  See \cite{PoCIT}, Theorem 2.6.
\end{proof}

\subsection{Crystalline extensions}
In this subsection we consider $(\Ph,\Gamma)$-modules over
the Robba ring $\CR_E.$ 
As usual, the first cohomology group $H^1(\Qp, \bD)$ can be interpreted in terms of extensions.
Namely, to any cocycle $\alpha =(a,b) \in Z^1(C_{\Ph,\gamma}(\bD))$ one associates 
the extension
\[
0\rightarrow \bD \rightarrow \bD_{\alpha} \rightarrow \CR_E\rightarrow 0
\]
such that $\bD_{\alpha}=\bD \oplus \CR_Ee$ with $\Ph (e) =e+a$ and $\gamma (e)=e+b.$
This defines a canonical  isomorphism
\[
H^1(\Qp, \bD) \simeq \text{Ext}^1 (\CR_E, \bD).
\]
We say that $\cl (\alpha)\in H^1(\Qp,\bD)$ is crystalline if 
\[
\dim_E \CDcris (\bD_{\alpha})=\dim_E \CDcris (\bD)+1
\]
and define
\[
H^1_f(\Qp,\bD)= \{\cl (\alpha)\in H^1(\Qp,\bD) \,\mid \, \cl (\alpha) \,\,\text{\rm is crystalline}\,\}.
\]
It is easy to see that $H^1_f(\Qp,\bD)$ is a subspace of $H^1(\Qp,\bD).$ If $\bD$ is semistable
(even potentially semistable), the equivalence 
\eqref{equivalence Dst and Dcris for phi-gamma mod} between 
the category of semistable $(\Ph,\Gamma)$-modules and 
filtered $(\Ph,N)$-modules allows to compute 
$H^1_f(\Qp,\bD)$ in terms of $\CDst (\bD).$ This gives a canonical
exact sequence
\begin{equation}
\label{exact sequence with H_f}
0\rightarrow H^0(\Qp, \bD)\rightarrow \CDcris (\bD)\rightarrow 
\CDcris (\bD)\oplus t_{\bD}(\BQ_p) \rightarrow H^1_f(\Qp, \bD)
\rightarrow 0,
\end{equation}
where $ t_{\bD}(\Qp)=\CDst (\bD)/\F^0\CDst (\bD)$ is the tangent 
space of $\bD$ (\cite{Ben11}, Proposition 1.4.4 and \cite{Ne92}, Sections 
1.19-1.21). In particular, one has 
\begin{eqnarray}
\label{dimension of H_f}
&&H^0(\Qp, \bD)=\F^0\CDst (\bD)^{\Ph=1, N=0}, \nonumber\\
&&\dim_E H^1_f(\Qp, \bD) = \dim_E t_{\bD}(\Qp) +\dim_L H^0(\Qp, \bD)   
\end{eqnarray}
(see \cite{Ben11}, Proposition 1.4.4 and Corollary 1.4.5). Moreover, 
$H^1_f(\Qp,\bD)$ and $H^1_f(\Qp, \bD^*(\chi))$
are orthogonal complements  to each other under duality 
\eqref{duality for coh phi-gamma modules} (\cite{Ben11}, Corollary 1.4.10).

We will call exponential map the connecting map in \eqref{exact sequence with H_f}
\begin{equation}
\label{exp for phi-gam mod}
\exp_{\bD}\,\,:\,\, t_{\bD}(\BQ_p) \rightarrow H^1_f(\Qp, \bD).
\end{equation}
Let $V$ be a potentially semistable  representation of $G_{\Qp}.$ 
Let 
\[
t_{V}(\BQ_p)=\Dst (V)/\F^0\Dst (V)
\] 
denote the tangent 
space of $V$ and let  $H^1_f(\Qp,V)$ be  the subgroup of $H^1(\Qp,V)$  defined 
by \eqref{Bloch-Kato local conditions}.
In \cite{BK}, Bloch and Kato constructed a map
\begin{equation}
\label{exp for representations}
\exp_{V}\,\,:\,\, t_{V}(\BQ_p) \rightarrow H^1_f(\Qp, V)
\end{equation}
using the fundamental exact sequence relating the 
rings of $p$-adic periods $\Bc$ and $\Bd$.
From Theorem~\ref{berger theorem}
it follows  that
$t_V(\Qp)$ is canonically isomorphic to $t_{\Ddagrig (V)}(\Qp)$
and that 
\[
H_f^1(\Qp,\Ddagrig (V))\simeq H^1_f(\Qp,V)
\]
(see \cite{Ben11}, Proposition 1.4.2). Moreover 
the diagram
\begin{displaymath}
\xymatrix{
t_{V}(\BQ_p) \ar[rr]^{\exp_V} \ar[d]^{=} & &H^1_f(\Qp, V) \ar[d]^{=}\\
t_{\Ddagrig (V)}(\BQ_p) \ar[rr]^{\exp_{\Ddagrig (V)}}& &H^1_f(\Qp, \Ddagrig (V)).
}
\end{displaymath}
commutes (\cite{Ben14}, Section 2). Therefore, our definition of the
exponential map \eqref{exp for phi-gam mod} agrees with
the Bloch--Kato's one.

\subsection{Cohomology of isoclinic modules}
The results of this section are proved in \cite{Ben11} (see Proposition 1.5.9 and Section 1.5.10 of \textit{op. cit.}).
 Let $\bD$ be semistable $(\Ph,\Gamma)$-module over $\CR_E$ of rank $d$.
Assume  that $\CDst (\bD)^{\Ph=1}=\CDst (\bD)$ and  that 
the all Hodge-Tate weights of $\bD$ are $\geqslant 0.$ Since $N\Ph =p\Ph N$ this implies  that
$N=0$  on $\CDst (\bD)$ and $\bD$ is crystalline.

The canonical map $\bD^{\Gamma} \rightarrow \CDcris (\bD)$ is an isomorphism and therefore  
\[
H^0(\Qp, \bD)\simeq \CDcris (\bD)= \bD^{\Gamma}
\]
is a $E$-vector space of dimension $d.$ 
The Euler-Poincar\'e characteristic formula  gives
\[
\dim_E H^1(\Qp, \bD) =d +\dim_E H^0(\Qp, \bD)+\dim_E H^0(\Q_p,\bD^*(\chi)) =2d.
\] 
On the other hand  $\dim_EH^1_f(\Qp, \bD)=d$ by 
\eqref{dimension of H_f}. 
The group $H^1(\Qp, \bD)$ has the following explicit description. The map 
\begin{eqnarray}
&&i_{\bD} \,:\, \CDcris (\bD) \oplus \CDcris (\bD) \rightarrow  H^1(\Qp, \bD),
\nonumber\\
&&i_{\bD} (x,y) =\cl (-x,\, \log \chi (\gamma)\, y)
\nonumber
\end{eqnarray}
is an isomorphism. (Remark that the sign $-1$ and   $\log \chi (\gamma)$ are normalizing factors.) 
We let denote
$i_{\bD,f}$ and $i_{\bD,c}$ the restrictions of $i_{\bD}$ on the first and second summand respectively.
Then $\im (i_{\bD,f})=H^1_f(\Qp,\bD)$ and we set 
$H^1_c(\Qp, \bD)=\text{Im} (i_{\bD,c}).$ Thus we have a canonical decomposition
\begin{equation}
H^1(\Qp, \bD) \simeq H^1_f(\Q_p, \bD)\oplus H^1_c(\Q_p, \bD).
\end{equation}

Now consider the dual module $\bD^*(\chi)$. It is crystalline, 
\linebreak
$\CDcris (\bD^*(\chi))^{\Ph=p^{-1}}=\CDcris (\bD^*(\chi))$ and 
the  Hodge-Tate weights of $\bD^*(\chi)$ are all $\leqslant 0.$  Let
\[
[\,\,,\,\,]_{\bD}\,:\,\CDcris (\bD^*(\chi))\times \CDcris (\bD) \rightarrow E
\] 
denote the canonical pairing. Define  
\[
i_{\bD^*(\chi)} \,:\, \CDcris (\bD^*(\chi)) \oplus \CDcris (\bD^*(\chi)) \rightarrow  H^1(\Qp, \bD^*(\chi))
\]
by
\[
i_{\bD^*(\chi)} (\alpha,\beta) \cup i_{\bD} (x,y)= [\beta ,x]_{\bD} -[\alpha, y]_{\bD}.
\]
As before, let $i_{\bD*(\chi),\,f}$ and $i_{\bD^*(\chi),\,c}$ denote  the restrictions of $i_{\bD}$ on the first and second summand respectively. From $H^1_f(\Qp, \bD^* (\chi))=H^1_f(\bD)^{\perp}$ it follows  that  $\text{Im}(i_{\bD^*(\chi),\,f})=H^1_f(\Qp, \bD^*(\chi))$
and   we set
\linebreak
 $H^1_c(\Qp, \bD^*(\chi))=\text{Im} (i_{\bD^*(\chi),\,c})$.
Again we have  a  decomposition 
\[
H^1(\Qp, \bD^*(\chi)) \simeq H^1_f(\Qp,\bD^*(\chi))\oplus H^1_c(\Qp,\bD^*(\chi)).
\]

Now we will compute the Iwasawa cohomology of some isoclinic 
modules.
Recall that the   well known computation 
of universal norms of $\Qp (1)$ gives:
\begin{equation}
\label{universal norms of Q(1)}
H^1_{\Iw}(\Qp,\Qp (1))_{\Gamma}=p^{\mathbf Z}
\end{equation}
under the Kummer map. 
The following proposition generalises this 
result to $(\Ph,\Gamma)$-modules.
It will be used in the proof of 
Proposition~\ref{Triviality of H^1_Iw} below.

\begin{myproposition}
\label{coinvariants od isoclinic modules}
 Let $\bD$ be an isoclinic $(\Ph,\Gamma)$-module
such that 
\linebreak
$\CDcris (\bD)^{\Ph=p^{-1}}=\CDcris (\bD)$ and
$\F^0\CDcris (\bD)=0.$ Then
\[
H_{\Iw}^1(\bD)_{\Gamma}=H^1_c(\Qp,\bD).
\]
\end{myproposition}
\begin{proof} For any continuous character $\delta\,:\,\Qp^* \rightarrow E^*$ let 
$\CR_E(\delta)$ denote the
$(\Ph,\Gamma)$-module of rank one $\CR_E e_{\delta}$ such that  
$\Ph (e_{\delta})=\delta (p)e_{\delta}$
and $\gamma (e_{\delta})=\delta (\chi (\gamma))e_{\delta}.$ 
In \cite{Cz08}, Proposition 3.1 Colmez proved that any $(\Ph,\Gamma)$-module of rank $1$ 
is isomorphic to a unique module of the form $\CR_E(\delta)$ for 
some $\delta.$

Set $\bD_m=\CR_E(\vert x\vert x^m).$ Let $e_m$ denote the canonical generator 
of $\bD_m.$   It is not difficult to see that an isoclinic $(\Ph,\Gamma)$-module $\bD$ which satisfies the conditions of Proposition~\ref{coinvariants od isoclinic modules}  is isomorphic to  
the direct sum $\underset{i=1}{\overset{d}\oplus} \bD_{m_i}$
for some $m_i\geqslant 1$  (\cite{Ben11}, Proposition 1.5.8). Therefore we can assume that 
$\bD=\bD_m$ for some $m\geqslant 1.$  The $E$-vector space $H^1(\bD_m)$
is generated by the cohomology classes
$\cl (\alpha_m)$ and $\cl (\beta_m)$
where
\begin{eqnarray}
&&\alpha_m=  \partial^{m-1} \left (\frac{1}{X}+\frac{1}{2},a \right )\, e_{m} ,\quad 
(1-\Ph)\,a=(1-\chi (\gamma) \gamma)\,\left (\frac{1}{X}+\frac{1}{2} \right )
\nonumber \\
&&\beta_m=
\partial^{m-1} \left (b,\frac{1}{X} \right )\,e_{m},\quad
(1-\Ph)\,\left (\frac{1}{X}\right )\,=\,(1-\chi (\gamma)\,\gamma)\,b
\nonumber
\end{eqnarray}
where  $\partial=(1+X)\,d/dX $
(\cite{Cz08}, Sections 2.3-2.5). Moreover 
$\cl (\alpha_m) \in H^1_f(\bD_m)$ and 
$\cl (\alpha_m) \in H^1_c(\bD_m)$ (\cite{Ben11}, Theorem 1.5.7). 
For $m=1$ the module $\bD_1=\CR_E(\chi)$ is \'etale and corresponds 
to the $p$-adic representation $E(1).$ 
Using the formula
\[
\frac{1}{X}=\underset{i=0}{\overset{p-1}\sum} (1+X)^i\Ph \left (\frac{1}{X}\right )
\]
it can be checked directly that $\psi (1/X)=1/X.$ Thus  
$\psi \left (\frac{1}{X}e_1 \right )=\frac{1}{X}e_1$ and we proved 
that $ \frac{1}{X}e_1\in \bD_1^{\psi=1}.$ An easy induction using the
identity $\psi \partial=p\partial \psi$ shows that for each $m\geqslant 1$
one has  $ \partial^{m-1}\left (\frac{1}{X}\right )e_m\in \bD_m^{\psi=1}.$ Since the image of  $\partial^{m-1}\left (\frac{1}{X}\right )e_m$ under 
the map $H^1_{\Iw}(\bD_m) \rightarrow H^1_{\Iw}(\bD_m)_{\Gamma} \subset 
H^1(\bD_m)$ is $\beta_m$, this proves that 
$H^1_{\Iw}(\bD_m)_{\Gamma} \subset H^1_c(\bD_m).$
On the other hand, Theorem~\ref{Iwasawa cohomology of phi-gamma mod} iii)
gives Hochschild-Serre  exact sequence 
\[
0\rightarrow H^1_{\Iw}(\bD_m)_{\Gamma} \rightarrow H^1(\bD_m)
\rightarrow H^2_{\Iw}(\bD_m)^{\Gamma}\rightarrow 0.
\]
By the same Theorem $H^2_{\Iw}(\bD_m)$ is  finite dimensional over $E$ and therefore
\[
\dim_E  H^2_{\Iw}(\bD_m)^{\Gamma}=\dim_E  H^1(\bD_m)_{\Gamma}=
\dim_E  H^2(\bD_m)=1.
\]
Thus 
\[ 
\dim_E  H^1_{\Iw}(\bD_m)_{\Gamma}=\dim_E  H^1(\bD_m)-
\dim_E  H^2(\bD_m)=1
\]
and we proved that $\dim_E  H^1_{\Iw}(\bD_m)_{\Gamma}=\dim_EH^1_c(\bD_m).$
This gives the proposition.
\end{proof}

\noindent
{\bf Remark.} It can be shown that the image of $p\in \Qp^*$
under the Kummer map is  $(1-1/p)\log (\chi)\cl (\beta_1)$
(see \cite{Ben00}, Proposition 2.1.5).

\section{The main conjecture}

\subsection{Regular submodules} In this section  we apply the theory 
of $(\Ph,\Gamma)$-modules to Iwasawa theory of $p$-adic representations. 
In particular, we propose a conjecture 
which can be seen as a (weak) generalisation of both  Greenberg's \cite{Gr89} and Perrin-Riou's \cite{PR95}
Main Conjectures.  Fix a prime number $p$ and  a finite set $S$ of primes of $\Q$ 
containing $p.$ For each number field $F$ we denote by $G_{F,S}$ the Galois group
of the maximal algebraic extension of $F$ unramified outside $S\cup \{\infty\}.$
For any topological module $M$ equipped with a continuous action of $G_{F,S}$
we write $H^*_S(F,M)$ for the continuous cohomology of $G_{F,S}$ with coefficients in
$M.$  Fix   a finite extension $E$ of $\Qp$ which will play
the role of the coefficient field. Let $V$ be an $p$-adic representation of 
$G_{\Q,S}$ with coefficients in $E.$  Recall that for each 
$v\in S$ we denote by $H^1_f(\BQ_v,V)$ the subgroup of 
$H^1(\BQ_v,V)$ defined by \eqref{Bloch-Kato local conditions} 
and by $H^1_f(\BQ,V)$ the Bloch--Kato Selmer group \eqref{definition of Bloch-Kato Selmer}.

The  Poitou--Tate exact sequence gives an exact sequence 
\begin{multline}
\label{Poitou Tate with H_f}
0 \rightarrow H^1_f(\BQ, V)\rightarrow H^1_S(\BQ, V)\rightarrow \underset{v\in S}\bigoplus 
\frac{H^1(\BQ_v,V)}{H^1_f(\BQ_v,V)}\rightarrow H^1_f(\BQ,V^*(1))^*
\rightarrow \\
H^2_S(\BQ,V)\rightarrow \underset{v\in S}\bigoplus  H^2(\BQ_v,V) 
\rightarrow H^0_S(\BQ,V^*(1))^*\rightarrow 0 
\end{multline}
(see \cite{FP}, Proposition 2.2.1)
Together with the well known formula
for the Euler characteristic this implies that
\begin{multline}
\label{formula for dimensions of H_f}
\dim_{E} H^1_f(\Q,V) - \dim_{E} H^1_f(\Q,V^*(1)) -
\dim_{E} H^0_S(\Q,V) + \\
+\dim_{E} H^0_S(\Q,V^*(1)) =
\dim_{E} t_V(\Qp)-d_{+}(V).
\end{multline}
where $d_{\pm}(V)=\dim_{E} V^{c=\pm 1}$ and $c$ denotes
the complex conjugation.

In the rest of this section  we assume that $V$ 
satisfies the following conditions:

{\bf C1)} $H^0_S(\Q,V)=H^0_S(\Q, V^*(1))=0.$ 

{\bf C2)} $V$ is semistable at $p$ and $\Ph \,:\,\Dst (V) \rightarrow \Dst (V)$
is semisimple at $1$ and $p^{-1}$.

{\bf C3)} $\Dc (V)^{\Ph=1}=0.$

{\bf C4)} $V$ satisfies one of the following conditions

{\bf a)}  $H^1_f(\BQ, V^*(1))=0.$

{\bf b)}  $V$ is a self dual representation, i.e. it is equipped with a
non degenerate bilinear form 
$V\times V\rightarrow E(1).$
\newline
\,

In support of these assumptions we remark that  representations which we have in mind 
appear as $p$-adic realisations of pure motives over $\Q.$   
Let $X/\Q$ be a smooth proper variety having good reduction outside
a finite set $S$ of places containing $p.$  Consider the motives $h^i(X)(m)$ where   
$0\leqslant i\leqslant 2\dim (X)$ and $m\in \mathbf Z.$ Let $H^i_{p}(X)$ denote
the $p$-adic \'etale cohomology of $X.$ The $p$-adic realization of $h^i(X)(m)$ is 
$V=H^i_{p}(X) (m).$  The action
of $\Gal (\overline{\Q}/\Q)$ factors throught $G_{\Q,S}.$ 
Moreover the restriction of $H^i_p(X)$ on the decomposition  group at $p$ 
is  semistable  if $X$ has semistable reduction at $p$ (and potentially semistable in general) \cite{Fa89}, \cite{Ts}.    Poincar\'e duality and the hard
Lefschetz theorem give a canonical isomorphism
\begin{equation}
\label{duality etale cohomology}
H_p^i(X)^*\simeq H^i_p(X)(i)
\end{equation}
and therefore 
\[
V^*(1)\simeq V(i+1-2m).
\]
The motive $h^i(X)(m)$ is pure of weight  of $w=i-2m$ and eventually replacing 
$V$ by $V^*(1)$ one can assume that $w\leqslant -1.$
By the comparision isomorphism of the $p$-adic Hodge theory 
\cite{Ts}
$\Dst (H^i_p(X))$ is isomorphic to 
the $\log$-crystalline cohomology $H^i_{\log-\text{\rm cris}}(X/\Qp).$
The semisimplicity of the Frobenius action on  
$H^i_{\log-\text{\rm cris}}(X/\Qp)$ is a well known open question. 
The weight monodromy conjecture of Deligne--Jannsen \cite{Ja89} predicts that
the absolute values of the eigenvalues of $\Ph$ acting on $\Dc (H^i_p(X))=\Dst (H^i_p(X))^{N=0}$
are $\leqslant i/2$ and therefore   $\Dc (V)^{\Ph=1}$ should be $0$  if $w\leqslant -1.$ 
If $X$ has good reduction at $p$ we do not need the weight  monodromy conjecture
and the nullity of  $\Dc (V)^{\Ph=1}$  follows inconditionally from the result of Katz--Messing \cite{KM}.

Assume that  $w=-1.$  Then $i$ is odd,  $m=\displaystyle\frac{i+1}{2}$ 
and the isomorphism \eqref{duality etale cohomology} shows that $V$ is self dual
and therefore satisfies {\bf C4b)}. 

Assume that $w\leqslant -2.$ One expects that
\begin{equation}
H^1_f(\BQ, V)=0\qquad \text{\rm if  $w\leqslant -2$.} 
\end{equation}
This follows from conjectural properties of the category $\mathcal M\mathcal M$ of 
mixed motives over $\Q.$  Consider   Yoneda groups  
\[H^n(\Q,h^i(X)(m))=
\Ext^n_{\mathcal M\mathcal M}(\Q(0),h^i(X)(m)), \quad n=0,1
\]
 and denote by $H^1_f(\Q,h^i(X)(m))$
 the subgroup of extensions having
"good reduction".   It is conjectured that the $p$-adic realisation functor induces an isomorphism
\begin{equation}
\label{p-adic regulator}
R_p\,:\,H^1_f(\Q,h^i(X)(m)) \iso H^1_f(\Q, H^i_p(X)(m)), \qquad \forall 
m\in \mathbf Z
\end{equation}
(see \cite{Fo92}). In particular we should have an isomorphism
\[
H^1_f(\Q,h^i(X)(i+1-m)) \iso H^1_f(\Q, V^*(1)).
\]
On the other hand from the semisimplicity of the category of pure motives
\cite{Ja90} it follows  that   
$
H^1(\Q,M)=0
$
for any pure motive $M$ of weight $\geqslant 0.$ In particular,
$H^1(\Q,h^i(X)(i+1-m))$ should vanish  if $-i-2+2m\geqslant 0.$ 
Together with \eqref{p-adic regulator} this imples {\bf C4b)}.  To sum up, from the motivic point of view the conditions
{\bf C4a)} and {\bf C4b)} correspond to weight $\leqslant -2$ and 
 weight $-1$ cases respectively.  
We consider  these two  cases separately below. 
\newline

\noindent
{\bf The weight $\leqslant -2$ case.} 
In this subsection we assume that $V$ satisfies 
{\bf C1-3)} and {\bf C4a)}. From {\bf C3)} it follows that the exponential map 
$t_V(\Qp)\rightarrow H^1_f(\Qp,V)$
is an isomorphism and we denote by $\log_V$ its inverse. Compositing $\log_V$ with
the localization map $H^1_f(\Q,V)\rightarrow H^1_f(\Qp,V)$  we obtain a map
\[
r_V\,:\,H^1_f(\Q,V) \rightarrow t_V(\Qp).
\]
One expects that if $V=H^i_p(X)(m)$ with $m\geqslant i/2+1$ this map is related to the syntomic 
regulator $R_{\mathrm{syn}}$ {\it via} the commutaive diagram
\begin{displaymath}
\xymatrix{
H^1_f(\Q,h^i(X)(m)) \ar[dr]^{R_{\text{\rm syn}}} \ar[d]^{R_p} & \\
H^1_f(\Q, V) \ar[r]^{r_V} &t_V(\Qp).
}
\end{displaymath}
Our next assumption reflects the hope that the syntomic regulator is
an injective map. 

{\bf C5a)} The localization map
\[
\text{\rm loc}_p\,:\,H^1_f(\Q,V)\rightarrow H^1_f(\Qp,V)
\]
is injective. 

In support of this assumption  we remark that if 
$V=H^i_p(X)(m)$ then  {\bf C5a)} holds  for all $m\neq i/2,i/2+1$ if in addition we assume that 
\[
H^0(\Q_v,V)=0, \qquad  \forall \,
v\neq p
\]
(and therefore $H^1_f(\Q_v,V)=0$ for all $v\neq p$)
(\cite{Ja89}, Lemma 4 and Theorem 3). 

\begin{definition}[\sc Perrin-Riou] Assume that $V$ is a $p$-adic representation
which satisfies the conditions {\bf C1-3)}, {\bf C4a)} and {\bf C5a)}. 

i) A $(\Ph,N)$-submodule $D$ of $\Dst (V)$ is regular if $D\cap \F^0\Dst (V)=0$
and the map
\[
r_{V,D}\,:\, H^1_f(\Q,V)\rightarrow \Dst (V)/(\F^0\Dst (V)+D)
\]
induced by $r_V$ is an isomorphism.

ii) Dually, a $(\Ph,N)$-submodule $D$ of $\Dst (V^*(1))$ is regular if 

\[D+\F^0\Dst (V^*(1))=\Dst (V^*(1))
\]
 and the map
\[
D\cap \F^0\Dst (V^*(1))\rightarrow H^1_f(\Q,V)^*
\]
induced by the dual map  $r_V^*\,:\,\F^0\Dst (V^*(1))\rightarrow H^1_f(V)^*$ is an 
isomorphism. 
\end{definition}
\noindent
{\bf Remark.} Assume that $H^1_f(\Q,V)=H^1_f(\Q,V^*(1))=0.$ Then $D$ is regular 
if the canonical projection $D\rightarrow t_V(\Qp)$ is an isomorphism of vector
spaces and our definition agrees with the definition given in \cite{Ben11}.
\newline

\noindent
{\bf The  weight $-1$ case.} In this subsection we assume that $V$ satisfies 
the conditions {\bf C1-3)} and {\bf C4b)}.  The formula \eqref{formula for dimensions of H_f} gives
\[
\dim_E t_V(\Qp)=d_{+}(V).
\]

\begin{definition}[\sc{Perrin-Riou}] Assume that $V$ is a $p$-adic representation
which satisfies the conditions {\bf C1-3)} and  {\bf C4b)}. A 
$(\Ph,N)$-submodule $D$ of $\Dst (V)$ is regular if the canonical map
$D\rightarrow t_V(\Qp)$ is an isomorphism.
\end{definition}

In the both cases {\bf C4a)} and {\bf C4b)} from \eqref{formula for dimensions of H_f}
it follows immediately that for a regular submodule
$D$ one has
\begin{equation}
\label{dimension of regular submodule}
\dim_ED=d_+(V).
\end{equation}

\noindent
{\bf Remark.} Our definition of a regular module in the weight $-1$ case  is slightly different 
from Perrin-Riou's definition.  If we assume the non-degeneracy  of 
$p$-adic height pairings $\left <\,,\,\right >_{V,D}$  (see Section 4) then our 
regular modules are regular also in the sense of \cite{PR95} but 
the converse in not true. Here we lose in generality to avoid additional 
technical assumptions.

\subsection{Iwasawa cohomology} We keep previous notation and conventions. 
Let $\Q^{\cyc} =\underset{n\geqslant 1}\cup \BQ (\mu_{p^{n}})$
and  $\Gamma=\Gal (\Q^{\cyc}/\Q).$ 
Then $\Gamma=\Delta \times \Gamma_0$ where $\Delta=\Gal (\BQ (\mu_p)/\BQ)$
and $\Gamma_0=\Gal (\Q^{\cyc}/\BQ (\mu_p)).$
Set
$F_n=\BQ (\mu_{p^{n+1}})^{\Delta}$ and 
$F_{\infty}=\underset{n\geqslant 1}\cup F_n.$ The Galois group 
$\Gal (F_{\infty}/F)$ is canonically isomorphic to
$\Gamma_0.$ Let $\Lambda = \CO_E [[\Gamma_0]]$ denote 
the Iwasawa algebra of $\Gamma_0.$

Let $V$ be a $p$-adic representation of $G_{\Q,S}.$ Fix a $\CO_E$-lattice $T$ of
$V$ stable under the action of $G_{\Q,S}$ and consider the Iwasawa cohomology
\[
H^i_{\Iw,S}(\Q, T)=\underset{\text{\rm cores}}\varprojlim H^1_S(F_n,T),\qquad
H^i_{\Iw}(\Q_v,T)=H^i_{\Iw}(\Q_v, T)\otimes_{\CO_E}E.
\]

The main properties of these groups are  summarized below 
(see also \cite{PR92}).
\newline \,

i)  $H^i_{\Iw,S}(\Q,V)=0$ and $H^i_{\Iw}(\mathbf Q_v,T)=0$ if $i\ne 1,2;$
\newline
\,

ii) If $v\ne p$, then 
 $H^i_{\Iw}(\mathbf Q_v,T)$ are  finitely-generated
$\Lambda$-torsion modules. In particular, 
\begin{equation}
\label{Hiw is torsion if v not p}
H^1_{\Iw}(\mathbf Q_v,T)\simeq 
H^1(\mathbf Q_v^{\mathrm{ur}}/
\mathbf Q_v,(\Lambda \otimes_{\CO_E}T^{I_v})^\iota).
\end{equation}
\newline
\,

iii) If $v=p$ then $H^2_{\Iw}( \Qp,T)$ is a  finitely generated  
$\Lambda$-torsion module. 
Moreover
\begin{equation*}
\mathrm {rank}_{\Lambda} \left (H^1_{\Iw}(\mathbb  Q_p,T)\right )\,=\,d,\qquad
H^1_{\Iw}(\Qp,T)_{\mathrm{tor}}\simeq H^0(\mathbf Q_p(\zeta_{p^\infty})\,,T).
\end{equation*}
We remark that by local duality $H^2_{\Iw}(\mathbb Q_p,T)\simeq 
H^0(F_{\infty,p},V^*(1)/T^*(1))$.
\newline
\,

The $\Lambda$-module structure of  $H^i_{\Iw,S}(\Q, T)$  depends heavily on the following conjecture formulated 
by Greenberg \cite{Gr89}.  

\begin{leopoldt} Let $V$ be a $p$-adic representation of $G_{\Q,S}$
which is potentially semistable at $p$. Then
$$
H^2_S(F_\infty,V/T)=0.
$$
\end{leopoldt}

We have the following result, proved in \cite{PR95}, Proposition 1.3.2.

\begin{myproposition} Assume that the weak Leopoldt conjecture holds for $V.$
Then $H^2_{\Iw,S}(\Q,T)$ is  $\Lambda$-torsion and 

\begin{displaymath}\text{\rm rank}_{\Lambda}  H^1_{\Iw,S}(\Q,T)=
 d_{-}(V). 
\end{displaymath}
\end{myproposition} 

Passing to projective limits
in the Poitou-Tate exact sequence for $V$  one obtains an exact sequence
\begin{multline} 
\label{Poitou-Tate Iwasawa}
0\rightarrow H^2(F_{\infty}, V^*(1)/T^*(1))^{\wedge}
\rightarrow H^1_{\Iw,S}(\Q, T)\rightarrow 
\underset{v\in S}\bigoplus H^1_{\Iw}(\Q_v,T)\rightarrow \\
\rightarrow 
H^1_S(F_\infty,V^*(1)/T^*(1))^{\wedge} 
\rightarrow H^2_{\Iw,S}(\Q,T) 
\rightarrow \underset{v\in S}\bigoplus H^2_{\Iw}(\Q_v,T)\rightarrow
\\
 \rightarrow 
H^0_S(F_\infty,V^*(1)/T^*(1))^{\wedge}\rightarrow 0,
\end{multline}
where $(-)^{\wedge}$ denotes the Pontriagin dual. Therefore 
if the weak Leopoldt conjecture holds for $V^*(1)$ we have an exact sequence
\begin{multline} 
0
\rightarrow H^1_{\Iw,S}(\Q, T)\rightarrow 
\underset{v\in S}\bigoplus H^1_{\Iw}(\Q_v,T)
\rightarrow 
H^1_S(F_\infty,V^*(1)/T^*(1))^{\wedge} 
\rightarrow \\H^2_{\Iw,S}(\Q,T) 
\rightarrow \underset{v\in S}\bigoplus H^2_{\Iw}(\Q_v,T)
 \rightarrow 
H^0_S(F_\infty,V^*(1)/T^*(1))^{\wedge}\rightarrow 0.\nonumber
\end{multline}

\subsection{The complexes $\SC_{\Iw} (V,D)$ and $\SC (V,D)$} 
In this section we construct  Selmer complexes $\SC_{\Iw} (V,D)$ and $\SC (V,D)$
which play the central role in our approach to the Iwasawa theory.
Nekov\'a\v r's book  \cite{NeSC} provides a detailed study of Selmer
complexes associated to Greenberg's local conditions. 
For the  purposes of this paper one   should work with local conditions associated  to general $(\Ph,\Gamma)$-submodules 
of $\Ddagrig (V).$ In this context the general formalism
of Selmer complexes was developed by Pottharst in \cite{Po13} and we refer to \textit{op. cit.} and \cite{PoCIT}  for further information and details. 

Let $V$ be a $p$-adic representation of the Galois group $G_{\Q,S}$
with coefficients in $E.$ We fix a $G_{\Q,S}$-stable lattice $T$ of $V.$
Recall that we denote by  $C^{\bullet}(G_{\BQ,S}, (T\otimes \La)^{\iota})$
and $C^{\bullet}(G_{\BQ_v}, (T\otimes \La)^{\iota})$ the complexes
of continuous cochaines of  $G_{\BQ,S}$ and $G_{\BQ_v}$
respectively with coefficients in $(T\otimes \La)^{\iota}$ 
and  by $\RG_{\Iw,S}(\Q,T)$ and 
$\RG_{\Iw}(\Q_v,T)$ these complexes viewed as objects
of  $ D^{\mathrm b}(\La).$
Recall that Shapiro's lemma gives  canonical isomorphisms 
\begin{equation}
\R^i\Gamma_{\Iw,S}(\Q,T) \simeq H^i_{\Iw,S}(\Q, T), \qquad 
\R^i\Gamma_{\Iw}(\Q_p,T) \simeq H^i_{\Iw}(\Q_p, T).\nonumber
\end{equation}

The derived version of the Poitou--Tate exact sequence for Iwasawa cohomology  reads 
\[
\RG_{\Iw,S}(\Q,T) \rightarrow \underset{v\in S}\bigoplus \RG_{\Iw}(\Q_v,T)
\rightarrow \RG (F_{\infty},V^*(1)/T^*(1))^{\wedge}[-2]
\]
(see \cite{NeSC}, Proposition  8.4.22). Passing to cohomology in this
exact triangle one obtains \eqref{Poitou-Tate Iwasawa}.

Let $\overline V=V\otimes_{\BQ_p}\CH^{\iota}.$ Then
\begin{eqnarray}
&&\RG_{\Iw,S} (\Q,\overline V) \simeq 
\RG_{\Iw,S} (\Q, T)\otimes^{\mathbf L}_{\La}\CH, \nonumber\\
&&\RG_{\Iw} (\Q_v,\overline V) \simeq 
\RG_{\Iw} (\Q_v, T)\otimes^{\mathbf L}_{\La}\CH. \nonumber
\end{eqnarray}

Assume that $V$ is potentially semistable. To each
$(\Ph,N)$-submodule  $D$ of  $\Dst (V)$ we associate a  complex
$\SC_{\Iw} (V,D)$ which can be seen as a direct generalisation of 
Selmer complexes determined by Greenberg's local conditions and studies in 
Chapters 7-8 of \cite{NeSC}.
We define local conditions $U^{\bullet}_{\Iw,v}(V,D)$ ($v\in S$). For $v\neq p$ we 
set
\[
U_{\Iw,v}^{\bullet}(V,D)=\RG_{\Iw,f}(\Q_v,T)\otimes_{\Lambda}\CH
\]
where  
\[
\RG_{\Iw,f}(\Q_v,T)=\left [ T^{I_v}\otimes \Lambda^{\iota}
\xrightarrow{\Fr_v-1}T^{I_v}\otimes \Lambda^{\iota}\right ].
\]
Here $I_v$ denotes the inertia group at $v$, $\Fr_v$ is the geometric Frobenius  
and the first term is placed in degree $0$.
Let  $H^i_{\Iw,f}(\Q_v,T)$ denote the cohomology of this complex. 
From \eqref{Hiw is torsion if v not p} it follows 
that  $H^i_{\Iw,f}(\Q_v,T)=0$ for $i\neq 1$ and 
$H^1_{\Iw,f}(\Q_v,T)=H^1_{\Iw}(\Q_v,T)$ is  $\La$-torsion.

Now we define the local condition at $p.$
Let $\bD$ be the $(\Ph,\Gamma)$-submodule of $\Ddagrig (V)$  associated to $D$ by 
Theorem~\ref{berger theorem}. Set
\[
U_{\Iw,p}^{\bullet}(V,D)=\RG_{\Iw}(\bD)=\left [\bD^{\Delta} \xrightarrow{\psi-1}\bD^{\Delta} \right ].
\]
From Theorem~\ref{Iwasawa cohomology of phi-gamma mod} it follows that
we have a canonical map $U_{\Iw,p}^{\bullet}(V,D) \rightarrow 
\RG_{\Iw} (\Q_p, T)\otimes^{\mathbf L}_{\La}\CH.$
This gives a diagram in $\mathrm D^{\mathrm b}_{\mathrm{coad}} (\mathscr H)$
\begin{displaymath}
\xymatrix{
\RG_{\Iw,S}(\Q,T)\otimes^{\mathbf L}_{\Lambda}\CH \ar[r] &\underset{v\in S}\bigoplus
\RG_{\Iw}(\Q_v,T)\otimes^{\mathbf L}_{\Lambda}\CH\\
&\underset{v\in S}\bigoplus U^{\bullet}_{\Iw,v}(V,D) \ar[u]}
\end{displaymath}
and we denote by ${\SC}_{\Iw} (V,D)$ 
the  associated Selmer complex
\begin{multline}
\label{definition of SC_Iw}
{\SC}_{\Iw}(V,D)=
\text{\rm cone}
\left ( \left (\RG_{\Iw,S}(\Q,T)\otimes^{\mathbf L}_{\Lambda}\CH \right  ) 
\bigoplus \underset{v\in S}\bigoplus U^{\bullet}_{\Iw,v}(V,D)\right.
\rightarrow \\
\left.
\underset{v\in S}\bigoplus
\RG_{\Iw}(\Q_v,T)\otimes^{\mathbf L}_{\Lambda}\CH \right  )[1].
\end{multline}
Now we define the complex $\RG (V,D).$ For $v\neq p$ let 
\[
\RG_f(\Q_v,V)= \left [V \xrightarrow{\mathrm{Fr}_v-1} V \right ]
\]
where the first term is placed in degree $0.$ It is clear that
$\R^0\Gamma_f(\Q_v,V)=H^0(\Q_v,V)$ and $\R^1\Gamma_f(\Q_v,V)$ 
coincides with the group $H^1_f(\Q_v,V)$ defined by \eqref{Bloch-Kato local conditions}.  Define 
\begin{equation}
U_v^{\bullet}(V,D)=\begin{cases} \RG_f (\Q_v, V), &\textrm{if $v\neq p$} \\
\RG (\Qp,\bD), &\textrm{if $v=p$}.
\end{cases}
\end{equation}
From Theorem~\ref{Iwasawa cohomology of phi-gamma mod} it follows that
\[
U_v^{\bullet}(V,D) =U_{\Iw,v}^{\bullet}(V,D)\otimes^{\mathbf L}_{\CH}E, \qquad v\in S. 
\]
Consider the Selmer complex $\SC (V,D)$ associated to  this data
\begin{equation}
\label{definition of SC}
\SC(V,D)=
\text{\rm cone}
\left ( \RG_{S}(\Q,T) 
\bigoplus \underset{v\in S}\bigoplus U^{\bullet}_{v}(V,D)
\rightarrow 
\underset{v\in S}\bigoplus
\RG (\Q_v,V) \right  )[1].
\end{equation}
We denote by $H^*_{\Iw}(V,D)$ and $H^*(V,D)$ the cohomology
of $\RG_{\Iw}(V,D)$ and $\RG(V,D)$ respectively.
The main properties of our Selmer complexes 
are summarized below.

\begin{myproposition}[{\sc Pottharst}]
\label{properties of SC_Iw} 
i) The complex $\SC_{\Iw}(V,D)$ 
has cohomology concentrated in degrees $[1,3]$ consisting
of coadmissible $\CH$-modules.

ii) One has
\[
\mathrm{rank}_{\CH}H^1_{\Iw}(V,D)=
\mathrm{rank}_{\CH}H^2_{\Iw}(V,D)
\]
Moreover 
$H^3_{\Iw}(V,D)=\left (T^*(1)^{H_{\Q,S}}\right )^*\otimes_{\La}\CH$
where $H_{F,S}=\Gal (\Q^{(S)}/F_{\infty}).$ 

iii) The complex $\SC(V,D)$ has cohomology concentrated in degrees $[0,3]$ consisting of of finite dimensional $E$-vector spaces and
\begin{equation}
\SC_{\Iw}(V,D) \otimes_{\CH}^{\mathbf L}E \simeq 
\SC (V,D). \nonumber
\end{equation}
In particular, we have canonical exact sequences
\begin{equation}
\label{Hochschild-Serre for SC} 
0\rightarrow H^i_{\Iw}(V,D)_{\Gamma} \rightarrow H^i(V,D)
\rightarrow H^{i+1}_{\Iw}(V,D)^{\Gamma}\rightarrow 0, \qquad i\in \mathbf N.
\end{equation}

iv) There are canonical dualities
\begin{eqnarray}
\label{duality for SC}
&&\Hom_E( \SC (V,D),E) \simeq \SC (V^*(1),D^{\perp})[3],\nonumber \\
&&\mathscr D \SC_{\Iw}(V,D)^{\iota} \simeq \SC_{\Iw}(V^*(1),D^{\perp})[3].
\nonumber
\end{eqnarray}
In particular, we have canonical exact sequences
\begin{equation}
\label{duality for H_Iw(V,D)}
0\rightarrow \mathscr D^1H^{4-i}_{\Iw}(V,D)\rightarrow 
H^i_{\Iw}(V^*(1),D^{\perp})^{\iota} \rightarrow \mathscr D^0H^{3-i}_{\Iw}(V,D)
\rightarrow 0.
\end{equation}

v) Assume that $D$ satisfies the following conditions:
\begin{eqnarray}
&&\F^0D=0 \quad  \textrm{and}\quad  \F^0(\Dst (V)/D)=\Dst (V)/D \nonumber\\
&&(\Dst (V)/D)^{\Ph=1,N=0}=0 \quad \textrm{and} \quad D/(ND +(p\Ph-1)D)=0. 
\nonumber
\end{eqnarray}
Then $H^1(\Qp,\bD)=H^1_f(\Qp,V)$ and  $H^1(V,D)=H^1_f(\Q,V).$
\end{myproposition}
\begin{proof} See \cite{PoCIT}, Theorem 4.1 and \cite{PoAF}, Proposition 3.7.
\end{proof}

\subsection{The Main Conjecture} In this subsection we use the
formalism of Selmer complexes to formulate a version of the Main Conjecture
of Iwasawa theory.
Let $M$ be a   pure motive over $\Q$ of weight $w\leqslant -1.$
Since the category of pure motives in semisimple, we can assume 
that $M$ is simple and $\neq \Q(m),$ $m\in \mathbf Z$.
Let $V$ denote the $p$-adic realisation of $M.$ Assume that 
$V$ is semistable. One expects that  for each regular $(\Ph,N)$-submodule  $D$ 
of $\Dst (V)$ there exists   a $p$-adic $L$-function of the form 
\[
L_p(M,D,s)= f_{D}(\chi (\g_0)^s-1),\qquad  f_{D}\in \CH
\]
interpolating algebraic parts of special values of the complex $L$-function $L(M,s)$ (see Section 0.2). 
Assume that  $V$ satisfies the conditions {\bf C1-4)} of Section 2.1. 
We propose the following conjecture.

\begin{mainconj} Let $M/\Q$ be a pure simple motive of weight $\leqslant -1.$
Assume that $M\neq \Q(m),$ $m\in \mathbf Z$.
Let $D$ be a regular
submodule of $\Dst (V).$ Then  

i)  $H^1_{\Iw}(V,D)=0,$

ii) $H^2_{\Iw}(V,D)$ is $\CH$-torsion and 
\[
\mathrm{char}_{\CH} (H^2_{\Iw}(V,D))=(f_{D}).
\]
\end{mainconj}

\noindent
{\bf Remarks.} 1) By Proposition~\ref{properties of SC_Iw} 
the nullity of $H^1_{\Iw}(V,D)$ implies that 
 $H^2_{\Iw}(V,D)$ is $\CH$-torsion.

2) By Proposition~\ref{properties coadmissible modules} any coadmissible 
$\CH$-torsion module $M$ decomposes into direct product 
$M\simeq \underset{i\in I}\prod \CH/\mathfrak p_i^{n_i}$ and 
from Lazard's theory \cite{La62} it follows that there exists a unique
up to multiplication by a unit of $\La [1/p]$ element $f\in \CH$
such that $\textrm{div}(f)= \underset{i\in I}\sum n_i\mathfrak p_i.$
The characteristic ideal  $\mathrm{char}_{\CH}M$ is defined to be 
the principal ideal generated by $f.$
 
3) The condition $M\neq \Q(m)$  implies that 
$(V^*(1))^{H_{\Q,S}}=0$ and therefore $H_{\Iw}^3(V,D)$
should vanish by Proposition~\ref{properties of SC_Iw}, ii). 
To sum up, we expect that under our assumptions the cohomology of $\SC_{\Iw}(V,D)$
is concentrated in degree $2.$  

4) Let $D^{\perp}$  denote the dual regular submodule. 
One can easily  formulate the Main Conjecture for the
dual pair $(M^*(1),D^{\perp}).$ Assume that $H^1_{\Iw}(V,D)=0.$  
From \eqref{duality for H_Iw(V,D)} it follows that
\[
H^2_{\Iw}(V^*(1),D^{\perp})^{\iota}\simeq 
H^2_{\Iw}(V,D).
\]
This isomorphism  
reflects  the conjectural functional equation
relating 
\linebreak
$L_p(M^*(1), D^{\perp},s)$ and 
$L_p(M, D,s).$ 

5) Assume that $V$ is ordinary at $p$ i.e. that the restriction of $V$ on
$G_{\Qp}$ is equipped with an encreasing exhaustive filtration  $F^iV$ such that
for all $i\in \mathbf Z$ the inertia group $I_p$ acts on 
$\gr^i (V)=F^iV/F^{i+1}V$ by $\chi^i.$ 
In \cite{Gr89}, Greenberg works with the Selmer group defined 
as follows. For each place $w$ of $F_{\infty}$ including the unique place
 above $p$  fix a decomposition group $H_w$ for $w$
in $H_{F,S}=\Gal ({\Q}^{(S)}/F_{\infty})$ and denote by  $I_w$
its inertia subgroup.  Set $A=V/T$ and $F^1A=F^1V/F^1T.$  
Define
\begin{equation}
H^1_{\mathrm{Gr}}(F_{\infty,w}, A)= \begin{cases} H^1(H_w/I_w, A^{I_w}), 
&\textrm{if $w\neq p$},\\
\ker \left (H^1(F_{\infty,p}, A) \rightarrow H^1(I_p, A/F^1A)\right )
&\textrm{if $w=p$}.
\end{cases}
\end{equation}
The Greenberg's Selmer group is defined as follows
\[
S(F_{\infty}, V/T)=\ker \left ( H^1_S(F_{\infty}, A) \rightarrow 
\underset{w\in S}\bigoplus \frac{H^1(F_{\infty,w}, A)}
{H^1_{\mathrm{Gr}}(F_{\infty,w}, A)}
\right ).
\]
It is well known that each semistable representation is ordinary 
and we set  $D=\Dst (F^1V).$ Then $\RG_{\Iw}(\bD)= \RG_{\Iw}(\Qp, F^1V)\otimes^{\mathbf L}_{\Lambda [1/p]}\CH$ and  directly from definition (\ref{definition of SC_Iw})
it follows that the  complex $\RG_{\Iw}(V,D)$ is isomorphic to 
$\widetilde \RG_{f,\Iw}(F_{\infty}/\Q,T)\otimes^{\mathbf L}_{\Lambda [1/p]}\CH$
where $\widetilde \RG_{f,\Iw}(F_{\infty}/\Q,T)$ denotes Nekov\'a\v r's Selmer complex associated to the local condition given by $F^1V.$ 
In \cite{NeSC}, Chapter 9, it is proved   that under some technical
conditions the characteristic ideal of the  Pontriagin dual 
$S(F_{\infty}, V^*(1)/T^*(1))^{\wedge}$ of 
$S(F_{\infty}, V^*(1)/T^*(1))$ coincides with the characteristic ideal of 
the second  cohomology group $\widetilde H^2_f(T)$ of 
$\widetilde \RG_{f,\Iw}(F_{\infty}/\Q,T).$ This allows to compare 
our conjecture to the Main Conjecture of \cite{Gr89}.

\section{Local structure of $p$-adic representations}

\subsection{Filtration associated to a regular submodule}
Let $D$ be a $(\Ph,N)$-submodule of $\Dst (V)$ such that 
$D\cap \F^0\Dst (V)=\{0\}.$  We associate to $D$ an increasing filtration
 $(D_i)_{i=-2}^2$ on $\Dst (V)$ setting
\begin{displaymath}
D_i\,=\,\begin{cases} 0 &\text{if $i=-2$,}\\
(1-p^{-1}\Ph^{-1})\,D+N(D^{\Ph=1}) &\text{if $i=-1$,}\\
D &\text{if $i=0$,}\\
D+\Dst (V)^{\Ph=1} \cap N^{-1}(D^{\Ph=p^{-1}})&\text{if $i=1$,}\\
\Dst (V) &\text{if $i=2$}.
\end{cases}
\end{displaymath}
It can be easily proved (see\cite{Ben11}, Lemma 2.1.5) that
$(D_i)_{i=-2}^2$ is the unique filtration on $\Dst (V)$
 such that

{\bf D1)} $D_{-2}=0$, $D_0=D$  and $D_2=\Dst (V);$

{\bf D2)} $(\Dst (V)/D_1)^{\Ph=1,N=0}=0$ and $D_{-1}= (1-p^{-1}\Ph^{-1})D_{-1}\,+N(D_{-1});$

{\bf D3)} $(D_0/D_{-1})^{\Ph=p^{-1}}=D_0/D_{-1}$ and $(D_1/D_0)^{\Ph=1}=D_1/D_0 .$ 
\newline
\,

\noindent
In addition, for the dual regular submodule  $D^{\perp}$ one has 
\[D^{\perp}_i=\text{\rm Hom}_{\Qp}(\Dst (V)/D_{-i}, \Dst(\Qp(1)) ).
\]
Let 
 $(D_i)_{i=-2}^2$ be the filtration  associated to $D$. By \eqref{equivalence Dst and Dcris for phi-gamma mod} it induces 
a filtration of $\Ddagrig (V)$ which we will denote by $(\bD_i)_{i=2}^2.$ 
Define
\[
\W=\bD_1/\bD_{-1},\quad \W_0=\gr_0 \Ddagrig (V),
\quad \W_1=\gr_0 \Ddagrig (V).
\]
Then we have an exact sequence
\begin{equation}
0\rightarrow \W_0 \rightarrow \W \rightarrow \W_1\rightarrow 0
\end{equation}
where by {\bf D1-3)}
\begin{eqnarray}
&&\CDst(\W_0)=\CDst(\W_0)^{\Ph=p^{-1}},\quad \F^0\CDst(\W_0)=0, \\
&&\CDst(\W_1)=\CDst(\W_1)^{\Ph=1}.
\end{eqnarray}

\begin{myproposition} 
\label{properties of filtration D_i}
Let $D$ be a regular submodule of $\Dst (V).$ Then

i) The canonical maps induce inclusions 
\[H^1(\bD_{-1}) \subset H^1_f(\bD)
\subset H^1(\bD_{1}) \subset H^1(\Qp,V)
\]

ii) One has $H^1_f( \bD_{-1})=H^1(\bD_{-1}).$

iii) The  sequences 
\begin{eqnarray}
&&0\rightarrow H^1(\bD_{-1}) \rightarrow H^1(\bD_1) 
\rightarrow H^1(\W)\rightarrow 0, \nonumber \\
&&0\rightarrow H^1(\bD_{-1}) \rightarrow H^1_f(\bD_1) 
\rightarrow H^1_f(\W)\rightarrow 0 \nonumber
\end{eqnarray}
are exact.
\end{myproposition}
\begin{proof} 1) By {\bf D2)} one has 
\[\Hom (D_{-1}, \Dst (\Qp (1)))^{\Ph=1, N=0}=0
\]
and by the Poincar\'e duality 
\begin{equation}
\label{nullity of H^2(D-1)}
H^2 ( \bD_{-1}) \simeq H^0 ( \bD_{-1}^*(\chi))^*=0.
\end{equation}
Now \eqref{Euler-Poincare phi gamma} and \eqref{dimension of H_f} implies
that $H^1_f( \bD_{-1})=H^1 ( \bD_{-1})$ and  ii) 
is proved. 

2) The exact sequence
\[
0\rightarrow \bD_1 \rightarrow \Ddagrig (V) \rightarrow \gr_2\Ddagrig (V)
\rightarrow 0
\]
gives an exact sequence
\[
0\rightarrow H^0 \left (\gr_2\Ddagrig (V)\right ) \rightarrow
H^1( \bD_1)\rightarrow  H^1(\Qp,V).
\]
From {\bf D2)} it follows that
\[
 H^0 \left (\gr_2\Ddagrig (V)\right )=
\F^0(\Dst (V)/D_1)^{\Ph=1,N=0}=0
\]
and the injectivity of $H^1( \bD_1)\rightarrow  H^1(\Qp,V)$
is proved. Since $\F^0\CDcris (\bD)=0$ and   $\CDcris (\bD)^{\Ph=1}=\Dc (V)^{\Ph=1}=0$ by
{\bf C3)}, the exact sequence \eqref{exact sequence with H_f}
induces a commutative diagram 
\begin{displaymath}
\xymatrix{
\CDcris (\bD) \ar[r]^{\exp} \ar[d] &H^1_f(\bD) \ar[d] \\
t_V(\Qp) \ar[r]^{\exp}  &H^1_f(\Qp,V).}
\end{displaymath}
where the exponential maps and the left vertical map are isomorphisms.
Thus $H^1_f(\bD)\rightarrow  H^1_f(\Qp,V)$ is an injection.
The same argument together with ii) proves that 
$H^1(\bD_{-1}) \subset H^1_f(\bD).$ This implies i).

3) Since the sequence
\[
0\rightarrow  H^1(\bD_{-1}) \rightarrow H^1(\bD_1) 
\rightarrow H^1(\W)\rightarrow H^2 ( \bD_{-1})
\]
is exact, iii) follows from ii) and \eqref{nullity of H^2(D-1)}.
\end{proof} 

Assume now that the canonical projection $\CDst (\bD) \rightarrow t_V(\Qp)$
is an isomorphism  i.e. that 
\begin{equation}
\label{Dst=D plus F^0}
\Dst (V)=D \oplus \F^0\Dst (V)
\end{equation}
as $E$-vector spaces. In this case the structure of $\W$ can be completely 
determined if we make the following  additional 
assumption  (see \cite{Gr94}, \cite{Ben11}).  

{\bf U)} The $(\Ph,\Gamma)$-module $\Ddagrig (V)$ has no saturated  \textit{crystalline} subquotient $U$ sitting in a non split  exact sequence of the form 
\begin{equation}
\label{definition U}
0 \rightarrow  \CR_E(\vert x\vert x^k)
\rightarrow U \rightarrow \CR_E (x^{m}) \rightarrow 0.
\end{equation}
\newline
\,

\noindent
We remark that if $k\leqslant m,$ then 
$H^1_f(\CR_E(\vert x\vert x^{k-m}))=0$ and there is no 
non trivial crystalline  extension \eqref{definition U}.
If $k>m,$ it follows from \eqref{dimension of H_f} that  
\linebreak
$\dim_EH^1_f(\CR_E(\vert x\vert x^{k-m}))=1$ and 
therefore there exists 
a unique (up to isomorphism) crystalline $(\Ph,\Gamma)$-module $U$
of the form \eqref{definition U}.

\begin{myproposition} 
\label{structure of W}
Let $D$ be a $(\Ph,N)$-submodule of $\Dst (V)$
which satisfies \eqref{Dst=D plus F^0}. Assume that the condition 
{\bf U)} holds. Then 

i) There exists a unique decomposition
\begin{equation}
\label{W=Mplus A_0plus A_1}
\W\simeq \A_0\oplus \A_1\oplus \bM 
\end{equation}
where $\A_0$ and  $\A_1$ are direct summands of  $\W_0$ and $\W_1$
of ranks 
\linebreak
$\dim_{E} H^0 (\W^*(\chi))$ and $\dim_{E}H^0(\W)$ respectively.
Moreover, $\bM$ is inserted in an exact sequence
\[
0\rightarrow \bM_0\xrightarrow{f} \bM \xrightarrow{g} \bM_1
\rightarrow 0
\]
where $\W_0\simeq \A_0\oplus \bM_0$,  $\W_1 \simeq \A_1\oplus \bM_1$  and
$\mathrm{rank} (\bM_0)=\mathrm{rank} (\bM_1).$

ii) One has 
\[
\dim_{E}H^1(\bM)=2e,\quad \dim_{E}H^1_f(\bM)=e,\quad\text{ where
$e=\mathrm{rank} (\bM_0)=\mathrm{rank} (\bM_1).$}
\]

iii) Consider the exact sequence
\[
0\rightarrow H^0(\bM_1) \xrightarrow{\delta_0} H^1(\bM_0)
\xrightarrow{f_1} H^1(\bM)
\xrightarrow{g_1} H^1(\bM_1)
\xrightarrow{\delta_1}H^2(\bM_0)
\rightarrow 0.
\]
Then $H^1(\bM_0) \simeq \text{\rm Im} (\delta_0)\oplus H^1_f(\bM_0),$ $\text{\rm Im} (f_1) =H^1_f (\bM)$ and 
\linebreak
$H^1(\bM_1) \simeq \text{\rm Im} (g_1) \oplus H^1_f(\bM_1).$
\end{myproposition}
\begin{proof} See \cite{Ben11}, Proposition 2.1.7 and Lemma 2.1.8.
\end{proof}

\subsection{The weight $\leqslant -2$ case}
We return to the study of the  cohomology of $p$-adic 
representations. Let $V$ is the $p$-adic realisation of a pure simple  motive
$M \neq \Q(m)$.  In this subsection we assume that $V$ satisfies the conditions 
{\bf C1-4a)} of Section 2.1. Let $D$ be a regular submodule of
$\Dst (V)$ and $\bD$ the associated $(\Ph,\Gamma)$-submodule
of $\Ddagrig (V).$
In the remainder of this paper we write 
$H^1(\bD)$ instead $H^1(\Qp,\bD)$ to simplify notation.

Define
\[
H^1_{f,\{p\}}(\Q, V)= \ker \left (
H^1_S(\Q,V) \rightarrow \underset{v\in S-\{p\}}\bigoplus
\frac{H^1(\Q_v,V)}{H^1(\Q_v,V)} \right ).
\]
Then \eqref{Poitou Tate with H_f}  gives an exact sequence 
\begin{equation}
0\rightarrow H^1_f(\Q,V)\rightarrow H^1_{f,\{p\}}(\Q,V)
\rightarrow \frac{H^1(\Q_p,V)}{H^1_f(\Qp,V)}
\rightarrow 0.
\end{equation}
From the regularity of $D$ we have a decomposition
\[
H^1_f(\Qp,V)= H^1_f(\Q,V) \oplus H^1_f(\bD)
\]
and therefore the restriction map induces an isomorphism
\[
H^1_{f,\{p\}}(\Q,V) \simeq 
\frac{H^1 (\Qp,V)}{H^1_f(\bD)}.
\]
Let $H^1_D(\Q,V)$ denote the inverse image of 
$H^1(\bD_{1})/H^1_f(\bD)$ under this isomorphism.
By Proposition 5  iii) we have an injection
\begin{equation}
\kappa_D\,:\,H^1_D(\Q,V) \hookrightarrow H^1(\W) 
\end{equation}
Since by {\bf D3)} $\CDcris (\W_0)^{\Ph=p^{-1}}= \CDcris (\W_0)$ 
and $\F^0\CDcris (\W_0)=0,$ the results of Section 1.8 give a canonical
decomposition 
\begin{equation}
\label{decomposition of W_0}
H^1(\W_0)=H^1_f(\W_0)\oplus H^1_c(\W_0).
\end{equation}

\begin{myproposition}
\label{Triviality of H^1_Iw}
 Assume that the weak Leopoldt conjecture holds for $V^*(1)$ and that
 
  a) $H^0(\W)=0.$ 
 
 b) The map $H^1_c(\W_0)\rightarrow H^1(\W)$ is injective.

c) One has 
\begin{equation}
\label{condition for triviality of H^1_Iw}
 \im (\kappa_D) \cap H^1_c(\W_0)  =\{0\} \quad \mathrm{in}\quad H^1(\W).
\end{equation}
Then $H^1_{\Iw}(V,D)=0$ and therefore $H^2_{\Iw}(V,D)$
is $\CH$-torsion.
\end{myproposition}
\begin{proof} The reader can compare this proof to the proof
of Theorem 5.1.3 of \cite{Ben14}. We will use the following
elementary lemma. 

\begin{mylemma}
\label{lemmaAB}
 Let $A$ and $B$ be two submodules of a finitely-generated 
free $\mathcal H$-module $M$.
Assume that the  natural maps $A_{\Gamma_1} \rightarrow  M_{\Gamma_1}$ and
 $B_{\Gamma_1} \rightarrow  M_{\Gamma_1}$ are both
injective. Then $A_{\Gamma_1}\cap B_{\Gamma_1}=\{0\}$ implies that $A\cap B=\{0\}.$
\end{mylemma}
\begin{proof} This is Lemma 5.1.4.1 of \cite{Ben14}. 
\end{proof}

We prove Proposition~\ref{Triviality of H^1_Iw}.
Because the weak Leopoldt conjecture holds for $V^*(1)$
the group  $H^1_{\Iw,S}(\Q,T)$ injects into 
$\oplus_{v\in S} H^1_{\Iw}(\Q_v,T).$
 Since 
\linebreak 
$H^1_{\Iw,f}(\Q_v,T)=H^1_{\Iw}(\Q_v,T)$
for $v\neq p,$ by the definition of the complex $\SC_{\Iw}(V,D)$
one has 
\[
H^1_{\Iw}(V,D)=
\left ( H^1_{\Iw,S}(\Q,T)\otimes_{\La}\CH \right ) \cap H^1_{\Iw}(\bD)
\qquad \mathrm{in}\quad H^1_{\Iw}(\Qp, T)\otimes _{\La}\CH.
\]
The $\La$-torsion part of $H^1_{\Iw,S}(\Q,T)$ is isomorphic
to $T^{H_{\Q,S}}=0$ and therefore  $A=H^1_{\Iw,S}(\Q,T)\otimes_{\La}\CH$
is a free $\CH$-module. Moreover  $A_{\Gamma}\subset H^1_{f,\{p\}}(\Q,V)$
Set $B=H^1_{\Iw}(\bD)/H^1_{\Iw}(\bD)_{\CH-\mathrm{tor}}.$ 
The $\CH$-torsion part of $H^1_{\Iw}(\bD)$ is contained in
$T^{H_{\Qp}}\otimes_{\La}\CH$ and 
$\dim_E\left (V^{H_{\Qp}}\right )_{\Gamma}=\dim_E
\left (V^{H_{\Qp}}\right )^{\Gamma}=0$
and therefore  $B_{\Gamma}=H^1_{\Iw}(\bD)_{\Gamma}.$ 

We will prove that  $B_{\Gamma}=H^1_{\Iw}(\bD)_{\Gamma}$ injects into 
$H^1(\bD_1)\subset H^1(\Qp,V).$  The exact sequence
\begin{equation}
\label{exact sequence quotient W_0}
0\rightarrow \bD\rightarrow \bD_1 \rightarrow \W_1
\rightarrow 0 \nonumber
\end{equation}
gives rise to an exact sequence
$
0\rightarrow H^0(\W_1) \rightarrow H^1(\bD)  \rightarrow H^1(\bD_1)
$
and therefore it is sufficient to show that 
\[H^0(\W_1)\cap H^1_{\Iw}(\bD)_{\Gamma}=0\quad \textrm{in}\quad H^1(\bD).
\]
Set $Z=H^0(\W_1)\cap H^1_{\Iw}(\bD)_{\Gamma}.$
Let $f\,:\,H^1(\bD) \rightarrow H^1(\W_0)$ denote 
the map induced by the natural projection $\bD\rightarrow \W_0.$
Since $H^0(\W)=0,$ the exact sequence 
$0\rightarrow \W_0\rightarrow \W\rightarrow \W_1
\rightarrow 0$
 induces an injection $H^0(\W_1)\hookrightarrow H^1(\W_0).$
 Thus $Z\cap \ker (f)=\{0\}.$ On the other hand, 
 by Proposition~\ref{coinvariants od isoclinic modules}
 \[
 f(H^1_{\Iw}(\bD)_{\Gamma})\subset  H^1_{\Iw}(\W_0)=H^1_c(\W_0).
 \]
From the exact sequence
\[
0\rightarrow H^0(\W_1)\hookrightarrow H^1(\W_0)\rightarrow H^1(\W)
\]
and  the injectivity of $H^1_c(\W_0)\rightarrow H^1(\W)$ we obtain that  
$H^1_c(\W_0)\cap  H^0(\W_1)=\{0\}$ and therefore $f(Z)=0.$
This proves that $Z=\{0\}$ and the injectivity of the map
$H^1_{\Iw}(\bD)_{\Gamma} \rightarrow H^1(\bD_1).$

To complete the proof of Proposition,  by Lemma 1 it is enough to show that
\begin{equation}
\label{H^1_f inter H^1{Iw}}
 H^1_{f,\{p\}}(\Q,V) \cap H^1_{\Iw}(\bD)_{\Gamma}=0 \qquad
 \mathrm{in} \quad H^1(\Qp,V).
\end{equation}
Since $H^1_{\Iw}(\bD)_{\Gamma} \subset H^1(\bD_1)$ 
and the map $\kappa_D$ is injective, \eqref{H^1_f inter H^1{Iw}} 
is equivalent to 
\[
H^1_D(\Q,V) \cap H^1_{c}(\W_0)=0
\]  
and  
the Proposition is proved.
\end{proof}

We want to discuss the relationship of this result with 
the phenomenon of trivial zeros of $p$-adic $L$-functions
studied in \cite{Gr94}, \cite{Ben11}, \cite{Ben14}.
We consider two cases.

\subsubsection*{\bf a) The case $\Dst (V)^{\Ph=1}=0$}   Assume in addition that $\Dst (V)^{\Ph=1}=0.$ Then $\bD_1=\bD$ and 
$\W=\W_0.$  By Proposition \ref{properties of filtration D_i} iii) 
one has 
\[
H^1_D(\Q,V) \simeq H^1(\bD)/H^1_f(\bD) \simeq H^1(\W_0)/H^1_f(\W_0).
\]
Consider the commutative diagram
\begin{displaymath}
\xymatrix{
\CDcris (\W_0) \ar[r]^{\overset{i_{D,f}}{\sim}} &H^1_f(\W_0)\\
H^1_D(\Q,V) \ar[u]^{\rho_{D,f}} \ar[r] \ar[d]_{\rho_{D,c}} 
&H^1(\W_0) \ar[u]_{p_{D,f}}
\ar[d]^{p_{D,c}}
\\
\CDcris (\W_0) \ar[r]^{\overset{i_{D,c}}{\sim}} &H^1_c(\W_0),}
\end{displaymath}
where $p_{D,f}$ and $p_{D,c}$ are projections given by
\eqref{decomposition of W_0} and  
 $\rho_{D,f}$ and $\rho_{D,c}$ are defined as the unique maps making
this diagram commute.

\begin{definition}[ see \cite{Ben14}]  
Let $V$ be a $p$-adic representation which satisfies the conditions
{\bf C1-4a)}.
Assume that $\Dst (V)^{\Ph=1}=0.$ Let $D$ be a regular submodule of 
$\Dst (V).$ The determinant
\begin{equation}
\ell (V,D)= \det \left ( \rho_{D,f} \circ \rho^{-1}_{D,c}\,\mid \,
\CDcris (\W_0) \right )
\nonumber
\end{equation}
will be called  the $\mathscr L$-invariant associated to $V$ and $D$.
\end{definition}

\subsubsection*{\bf b) The case $H^1_f(\Q,V)=H^1_f(\Q,V^*(1))=0$} Assume  that
\begin{equation}
\label{H^1(V)=H^1(V^*(1)}
H^1_f(\Q,V)=H^1_f(\Q,V^*(1))=0
\end{equation}
and that in addition $V$ satisfies 
\newline
\,

{\bf M)} The condition {\bf U)} of Section 3.1 holds  
and in the decomposition 
\eqref{W=Mplus A_0plus A_1}  $\A_0=\A_1=0.$
\newline
\newline
Note that the  typical example we have in mind   is $W_f(k)$ where $W_f$ is the $p$-adic representation associated
to a split multiplicative newform  $f$ of weight $2k$ on $\Gamma_0(N)$  with $p\mid N.$ 

From Proposition~\ref{structure of W} 
one has $H^1_f(\bM_0)=H^1_f(\bM)$ and
$H^1(\bM)/H^1_f(\bM_0) \simeq 
H^1 (\bM_1)/H^1_f(\bM_1).$
Thus $H^1_D(\Q,V) \simeq H^1 (\bM_1)/H^1_f(\bM_1)$
is a $E$-vector space of dimension $e=\mathrm{rank} (\bM_0)=
\mathrm{rank} (\bM_1)$ and 
one has a commutative diagram

\begin{equation}
\label{definition of L-invariant sequence}
\footnotesize
\xymatrix{
0 \ar[r] & H^0(\bM_1)\ar[r]^{\delta_0} & H^1(\bM_0) \ar[r]^{f_1} & H^1(\bM) \ar[r]^{g_1} & H^1(\bM_1) \ar[r]^{\delta_1} & H^2(\bM_0) \ar[r] & 0\\
& & & H^1_D(\Q,V)\ar[u]^{\kappa_D} \ar[ur]^{\bar \kappa_D} &  & &
}
\end{equation}
\normalsize
where the map  $\bar \kappa_D$ is injective.
As $H^0(\bM)=H^2(\bM)=0,$ the upper row is exact. This implies that $\im (g_1)$
is a $E$-vector space of dimension $e.$ Thus 
$
\im (\bar \kappa_D) =\im (g_1)
$
and again one has  a commutative diagram

\[
\xymatrix{
 &\CDcris (\bM_1) \ar[r]^{\overset{i_{D,f}}\sim} & H^1_f(\bM_1)\\
H^1_D(\Q,V) \ar[r]^{\sim} &\im (g_1)\ar[u]^{\rho_{M,f}} \ar[r] \ar[d]_{\rho_{M,c}} &
H^1(\bM_1) \ar[u]_{p_{D,f}}
\ar[d]^{p_{D,c}}
\\
&\CDcris (\bM_1) \ar[r]^{\overset{i_{D,c}}\sim} &H^1_c(\bM_1)}
\]

\begin{definition}[ see \cite{Ben11}]  
Let $V$ be a $p$-adic representation which satisfies the conditions
{\bf C1-3)} and \eqref{H^1(V)=H^1(V^*(1)}.
 Let $D$ be a regular submodule of 
$\Dst (V).$ Assume that  the condition {\bf M)} holds.  The determinant
\begin{equation}
\ell (V,D)= \det \left ( \rho_{D,f} \circ \rho^{-1}_{D,c}\,\mid \,
\CDcris (\bM_1) \right )
\nonumber
\end{equation}
will be called  the $\ell$-invariant associated to $V$ and $D$.
\end{definition}

In particular, in this case the $\ell$-invariant  is local, i.e. depends only 
on the restriction of the representation $V$ on a decomposition group at $p$.
\newline
\,

{\bf Remark.}  In \cite{Ben11} the $\ell$-invariant is defined 
in a slightly more general situation where only $\W_0$ vanishes. 
We do not include it here to avoid additional technical complications
in the formulation of our results.
\newline
\,

The following conjecture was formulated in \cite{Ben11} and \cite{Ben14}.

\begin{conjecture}[{\sc Extra zero conjecture}] 
Let $V$ be a $p$-adic representation which satisfies the conditions {\bf C1-4a)}.
Let $D$ be a regular subspace of $\Dst (V)$ and
let $e=\mathrm{rank} (\W_0).$ 
Then in both cases {\bf a)} and {\bf b)} the $p$-adic $L$-function $L_p(V,D,s)$ has a zero of order $e$ at $s=0$ and
\[
\underset{s\to 0}{\lim}\,
\frac{L_p(V,D,0)}{s^e}\,=\,
\ell (V,D)\,\mathcal{E}^+(V,D)\,
{\Omega_{p}(M,D)}
\frac{L(M,0)}{\Omega_{\infty}(M)}
\]
where $\mathcal{E}^+(V,D)$ is obtained from  $\mathcal{E}(V,D)$ 
by excluding zero factors. 
\end{conjecture}
Recall (see Section 0.2) that $\Omega_{p}(M,D)$ denote the determinant 
of the regulator map $r_{V,D}.$ 
Note that $\Omega_{p}(M,D)=1$ if $H^1_f(\Q,V)=H^1_f(\Q,V^*(1))=0.$
We refer to \cite{Ben11}, \cite{Ben13} \cite{Ben14} for precise formulation
of this conjecture  and a survey of known cases. Note that the  non vanishing of $\ell (V,D)$ is 
a difficult open problem which is solved only for Dirichlet  motives 
$\Q (\eta)$ \cite{FG}
and for elliptic curves \cite{SaintEtienne}. The following result shows 
that it is closely related to the expected vanishing of $H^1_{\Iw}(V,D).$

\begin{myproposition} 
\label{l-invariant and triviality Hiw}
Let $V$ be a $p$-adic representation which satisfies the conditions {\bf C1-3)}. Assume that the weak Leopoldt conjecture holds for $V^*(1).$
Then in the both cases {\bf a)} and {\bf b)}
above  the non vanishing of  $\ell (V,D)$ implies that 
$H^1_{\Iw}(V,D)=0$ and therefore $H^2_{\Iw}(V,D)$ in $\CH$-torsion. 
\end{myproposition}
\begin{proof} In the case {\bf a)} $\W_0=\W$ and the statement   follows directly from
Proposition~\ref{Triviality of H^1_Iw} and the definition of the 
$\ell$-invariant. In the case {\bf b)} the diagram 
\eqref{definition of L-invariant sequence} together with 
Proposition~\ref{structure of W} show that   the injectivity 
of $H^1_c(\bM_0) \rightarrow H^1(\bM)$ and the condition 
\eqref{condition for triviality of H^1_Iw} are both equivalent to
the following condition
\[
\im (\delta_0) \cap H^1_c(\bM_0)=\{0\}.
\] 
The statement follows now from \cite{Ben11}, Proposition 2.2.4
where it is proved that $\ell (V,D)$ can be computed as the slope 
of the map $\delta_0\,:\, H^0(\bM_1) \rightarrow H^1(\bM)$ 
with respect to the decomposition $H^1(\bM)\simeq H^1_f(\bM)\oplus H^1_c(\bM).$
\end{proof}

\section{The $p$-adic height pairing}

\subsection{The extended Selmer group}
Let $V$ be a 
 $p$-adic representation of $G_{\Q,S}$
which satisfies the conditions  {\bf C1-3)} of Section 2.1.
Fix a 
submodule $D$ of $\Dst (V)$ 
such that $\Dst (V)=D\oplus \F^0\Dst (V).$ 
We will always  assume that the condition {\bf M)} of Section 3.2 
holds for $D.$ Then by Proposition~\ref{structure of W} 
one has an exact sequence
\[
0\rightarrow \bM_0\rightarrow \bM \rightarrow \bM_1\rightarrow 0
\]
where $\CDcris (\bM_0)=D/D_{-1},$  $\CDcris (\bM_1)=D_1/D$ and
$\mathrm{rank} (\bM_0)=\mathrm{rank} (\bM_1).$

\begin{myproposition} There exists an exact sequence
\begin{equation}
\label{extended selmer}
0\rightarrow H^0(\bM_1)\rightarrow H^1 (V,D)\rightarrow H^1_f(\Q,V)
\rightarrow 0,
\end{equation}
where $H^1_f(\Q,V)$ is the Bloch--Kato's Selmer group and $\dim_E H^0(\bM_1)=\mathrm{rank}(\bM_1)$.
\end{myproposition}
\begin{proof} By the definition of $\RG (V,D)$ the  group $ H^1 (V,D)$
is the kernel of the map
\[
H^1_S(\Q,V)\oplus \left (\underset{v\in S-\{p\}}\bigoplus H^1_f(\Q_v,V)
\right ) \oplus H^1(\bD) \rightarrow \underset{v\in S}\bigoplus H^1(\Q_v,V).
\] 
The Proposition follows directly from this description of $ H^1 (V,D)$ 
together with the following facts

a) $H^0(\bM_1)=\ker (H^1(\bD) \rightarrow H^1(\Qp,V));$

b) $\im (H^1(\bD) \rightarrow H^1(\Qp,V))=H^1_f(\Qp,V).$

The proof of  a) and b) can be extracted from the construction
of the $\ell$-invariant in \cite{Ben11}, Section 2.2.1 
but we recall the arguments for reader's convenience. 
Consider the  exact sequence 
\[
0\rightarrow \bD_1\rightarrow \Ddagrig (V) \rightarrow \gr_2\Ddagrig (V)\rightarrow 0.
\]
In the proof of Proposition~\ref{properties of filtration D_i}
we saw that $H^0\left (\gr_2\Ddagrig (V)\right )=0$ and 
therefore the map $H^1(\bD_1)\rightarrow H^1(\Qp,V)$ is injective.
Taking the long cohomology sequence associated to
\[
0\rightarrow \bD \rightarrow \bD_1 \rightarrow \bM_1\rightarrow 0
\]
we obtain 
\[
0\rightarrow H^0(\bM_1)\rightarrow H^1(\bD) \rightarrow H^1(\bD_1)
\]
and a) is proved.

Using \eqref{exact sequence with H_f} it is not
difficult to show that  
and  $H^1_f\left (\gr_2\Ddagrig (V)\right )=0$ and therefore 
$H^1_f(\bD_1)\simeq H^1_f(\Qp,V).$ Consider the exact sequence 
\[
0\rightarrow \bD_{-1}\rightarrow \bD_1\rightarrow \bM \rightarrow 0.
\]
Since $H^0(\bM)=0$ and $H^1_f(\bD_{-1})=H^1(\bD_{-1}),$ 
by \cite{ben11}, Corollary 1.4.6  one has  an exact sequence 
\[
0\rightarrow H^1(\bD_{-1})\rightarrow H^1_f(\bD_1)\rightarrow H^1_f(\bM)
\rightarrow 0.
\]
On the other hand, from Proposition~\ref{structure of W} 
\[\im (H^1(\bM_0)\rightarrow H^1(\bM))=H^1_f(\bM).
\]
Thus $\im (H^1(\bD)\rightarrow H^1(\bD_1))=H^1_f(\bD_1)$ and
b) is proved.

\end{proof}

\begin{definition} We will call $H^1(V,D)$ the extended Selmer group 
associated to $D.$
\end{definition}

\noindent
{\bf Remarks} 1) Extended Selmer groups associated to ordinary local conditions 
were studied in \cite{Ne92}, \cite{NeSC}.

2) If $\bM=0$  the group $H^1(V,D)$ coincides with the Bloch--Kato's 
Selmer group. One expects that if $D$ is regular, the appearance  of $H^0(\bM_1)$ 
in the short exact sequence \eqref{extended selmer} reflects 
the presence of extra-zeros of the $p$-adic $L$-function $L_p(V,D,s).$  
We study this question  in \cite{Ben14b}.

\subsection{The $p$-adic height pairing}
In \cite{NeSC}, Nekov\'a\v r found a new construction 
of the $p$-adic height pairing on extended 
Selmer groups  defined by Greenberg's local conditions. 
In this section  we follow his approach {\it verbatim}
working with  local conditions defined by $(\Ph,N)$-submodules. 
We keep  notation and conventions of Section 4.1. 
Set $A=\CH/(X^2).$ Then also $A=E[X]/(X^2)$ and one has an exact sequence
\begin{equation}
\label{Bochstein for A}
0\rightarrow E\xrightarrow{X} A\rightarrow E\rightarrow 0.
\end{equation}
Set $V_A=V\otimes_EA^{\iota}$ and $D_A=D\otimes_EA^{\iota}$ and consider
the complex $\RG(V_A,D_A).$ The sequence \eqref{Bochstein for A}
induces    a distinguished triangle
\[
\RG(V_A,D) \rightarrow \RG(V_A,D_A) \rightarrow \RG(V,D)
\]
which gives the coboundary map 
\[H^1(V,D)\xrightarrow{\beta} H^2(V,D).
\]
Let $\bD^{\perp}\subset \Dst (V^*(1))$  denote the dual 
 submodule.
 
\begin{definition} Let $V$ be a $p$-adic representation which satisfies
the conditions {\bf C1-3)} and let $D$ be a $(\Ph,N)$-submodule of 
$\Dst (V)$ such that  the condition {\bf M)} holds for $D$. 
We define the $p$-adic height pairing associated 
to $D$ as the bilinear map  
\begin{equation}
\label{definition height}
\left < \,,\,\right >_{V,D}\,\,:\,\,H^1(V,D) \times H^1(V^*(1),D^{\perp})
\rightarrow E 
\end{equation}
given by
$
\left < x,y \right >_{V,D}=\beta (x) \cup y
$
where $\cup \,:\, H^2(V,D)\times H^1(V^*(1),D^{\perp})\rightarrow E$ 
denotes  the duality defined in Proposition~\ref{properties of SC_Iw}.
\end{definition}

\noindent
{\bf Remarks} 1) In the ordinary setting it is possible to work 
over $\Zp$ and Nekov\'a\v r's descent machinery gives
very general formulas of the Birch and Swinnerton-Dyer type. 
In the general  setting we are forced to work over $\Qp$ if we want to use 
the theory of $(\Ph,\Gamma)$-modules.  
  
2) The pairing \eqref{definition height} will be studied in detail in \cite{Ben14b}. In particular we  compare our construction to the 
$p$-adic height pairing constructed by Nekov\'a\v r in \cite{Ne92}
and  relate it to universal norms. 
\newline
\,

\indent
The following result can be seen as an analog of Proposition~\ref{l-invariant and triviality Hiw} in the weight $-1$ case.

\begin{myproposition} Let $V$ be a $p$-adic representation which satisfies 
the conditions {\bf C1-3)} and does not contain 
subrepresentations of the form $\Qp(m).$ Let $D$ be a $(\Ph,N)$-submodule of $\Dst (V)$ such that the condition {\bf M)} holds for $D.$ 
Assume that 

a) The weak Leopoldt conjecture 
holds for $V^*(1).$

b) The pairing 
$\left < \,,\,\right >_{V,D}$
is non degenerate.

Then $H^1_{\Iw}(V,D)=0$ and therefore  $H^2_{\Iw}(V,D)$
is $\CH$-torsion.

\end{myproposition} 
\begin{proof}
If the $p$-adic height  pairing is non degenerate, the map 
$\beta$ is injective. The diagram
\[
\xymatrix{
0\ar[r] &\CH \ar[r]^{X} \ar[d] &\CH \ar[r] \ar[d] &E \ar[r] \ar[d] & 0\\
0\ar[r] &E \ar[r]^{X}  &A \ar[r]  &E \ar[r]  & 0}
\] 
gives rise to a commutative diagram with exact rows
\[
\xymatrix{
0 \ar[r] &H^1_{\Iw}(V,D)_{\Gamma} \ar[r] \ar[d] &H^1(V,D) \ar[d]^{=} \ar[r] &H^2_{\Iw}(V,D)^{\Gamma} \ar[d]\\
 &H^1(V_A,D_A) \ar[r] &H^1(V,D) \ar[r]^{\beta} &H^2(V,D).}
\]
Thus $H^1_{\Iw}(V,D)_{\Gamma}=0.$ On the other hand,
as the weak Leopoldt conjecture holds for $V^*(1)$ and $V^{H_{\Q,S}}=0,$
the $\La [1/p]$-module $H^1_{\Iw,S}(\Q,V)$ is free and injects into
$H^1_{\Iw}(\Qp,V).$ Therefore, as in the proof of Proposition~\ref{Triviality of H^1_Iw} one has 
\[
H^1_{\Iw}(V,D)=\left (H^1_{\Iw,S}(\Q,T)\otimes_{\Lambda}\CH \right )\cap H^1_{\Iw}(\bD).
\]
 This implies that $H^1_{\Iw}(V,D)$ is $\CH$-free and the vanishing of 
$H^1_{\Iw}(V,D)_{\Gamma}$ gives $H^1_{\Iw}(V,D)=0.$ 
\end{proof}

\bibliographystyle{style}

\end{document}